\documentclass[11pt]{amsart}

\usepackage[T1]{fontenc}
\usepackage{amsmath,amssymb,mathtools}
\usepackage{microtype}
\usepackage[margin=1.08in]{geometry}
\usepackage{xcolor}
\usepackage{hyperref}

\hypersetup{
  colorlinks=true,
  linkcolor=blue!55!black,
  citecolor=blue!55!black,
  urlcolor=blue!55!black,
  pdftitle={On the splitting of surfaces in the motivic stable homotopy category},
  pdfauthor={Haoyang Liu},
  pdfsubject={Motivic stable homotopy and Milnor--Witt motives},
  pdfkeywords={motivic stable homotopy, Milnor--Witt motives, Chow--Witt correspondences, K3 surfaces}
}

\newtheorem{theorem}{Theorem}[section]
\newtheorem{lemma}[theorem]{Lemma}
\newtheorem{proposition}[theorem]{Proposition}
\newtheorem{corollary}[theorem]{Corollary}
\newtheorem{definition}[theorem]{Definition}

\theoremstyle{definition}
\newtheorem{remark}[theorem]{Remark}
\newtheorem{example}[theorem]{Example}
\numberwithin{equation}{section}

\newcommand{\SH}{\operatorname{SH}}
\newcommand{\DM}{\operatorname{DM}}
\newcommand{\DMW}{\widetilde{\operatorname{DM}}}
\newcommand{\CH}{\operatorname{CH}}
\newcommand{\Ch}{\operatorname{Ch}}
\newcommand{\Th}{\operatorname{Th}}
\newcommand{\Pic}{\operatorname{Pic}}
\newcommand{\Spec}{\operatorname{Spec}}
\newcommand{\Hom}{\operatorname{Hom}}
\newcommand{\End}{\operatorname{End}}
\newcommand{\im}{\operatorname{im}}
\newcommand{\id}{\operatorname{id}}
\newcommand{\one}{\mathbf 1}
\newcommand{\Z}{\mathbb Z}
\newcommand{\Q}{\mathbb Q}
\newcommand{\C}{\mathbb C}
\newcommand{\T}{\mathbb T}
\newcommand{\MW}{\mathrm{MW}}

\title{On the splitting of surfaces in motivic stable homotopy category}
\author{Haoyang Liu}
\address{Department of Mathematics, University of California,
Santa Barbara, CA 93106, USA}
\email{haoyangliu@ucsb.edu}

\makeatletter
\@namedef{subjclassname@2020}{%
  \textup{2020} Mathematics Subject Classification}
\makeatother

\begin{document}
\begin{abstract}
Over a perfect field of characteristic zero, we study when the top
cell splits off from a pointed smooth projective surface in motivic
stable homotopy, with K3 surfaces as the main example.
Adams--Riemann--Roch and the motivic Adams relations give a uniform
top-codimension annihilator. This yields an integral extreme-cell splitting in particular for complex projective K3 surfaces.
Chow--Witt correspondences also yield noncanonical three-term
decompositions of Milnor--Witt motives with extreme Tate summands and a
nilpotence criterion for lifting finite systems of Chow projectors.  In
positive characteristic, the annihilator statement holds after
inverting the characteristic.
\end{abstract}
\maketitle
\tableofcontents

\section{Introduction}
Stable splittings are a basic phenomenon in the classical stable
homotopy category.  Such splittings give geometric explanations for
decompositions in homology, cohomology, and other representable
theories.  Similar questions arise in motivic stable homotopy.  Let
$\SH(k)$ denote the motivic stable homotopy category over a field $k$.
For a smooth projective variety $X$ of dimension $d$, we call a
decomposition
\[
  X_+\simeq S^{0,0}\vee F\vee S^{2d,d}
\]
an \emph{extreme-cell splitting}.  A rational point already splits
off the bottom cell, so the problem is to split off the top cell.  The
obstruction is measured by the image of the reduced tangent class
under the motivic $J$-homomorphism
\[
  J_X:K_0(X)\longrightarrow\pi_0\Pic(\SH(X)),
  \qquad V\longmapsto[\Th_X(V)].
\]
Vanishing of
$J_X([T_X]-[\mathcal O_X^{\oplus d}])$ is equivalent to triviality of
the corresponding reduced tangent Thom object.  Motivic duality then
becomes a shifted self-duality of $X_+$ and produces the top retract.
The principal geometric result of this paper is an integral
extreme-cell splitting for complex projective K3 surfaces, together
with the exact order $24$ of the corresponding point $J$-class.

Abelian varieties provide the basic case in which the tangent bundle
is trivial.  Their integral extreme-cell splitting, together with finer
coefficient-localized decompositions over $\mathbb R$, was studied in
the author's earlier work~\cite[Theorems 3.6 and
7.4]{liu2024splittingabelianvarietiesmotivic}.  The present paper
instead treats nontrivial reduced tangent classes through the motivic
$J$-homomorphism and develops the corresponding Milnor--Witt
refinement.

The geometric theorem applies more generally to any pointed smooth
projective surface whose Chern classes satisfy the two
relations in the following statement.

\begin{theorem}\label{thm:intro-k3}
Let $k$ be a perfect field of characteristic zero, and let $X$ be a
smooth projective connected surface over $k$ with a point $o\in X(k)$.
If
\[
  c_1(T_X)=0,
  \qquad c_2(T_X)=24[o]
  \quad\text{in }\CH^*(X),
\]
then
\[
  \Th_X(T_X-\mathcal O_X^{\oplus2})\simeq\one_X
  \quad\text{in }\SH(X)
\]
and, for some $F_X\in\SH(k)$,
\[
  X_+\simeq S^{0,0}\vee F_X\vee S^{4,2}
\]
in $\SH(k)$.  If $k$ admits an embedding into $\C$, then
$J_X([\mathcal O_x])$ has exact order $24$ for every $x\in X(k)$.
In particular, every complex projective K3 surface admits such an
integral decomposition, and the point-class exponent is sharp on every
such surface.
\end{theorem}

For a perfect field of characteristic $p>0$, the same geometric
hypotheses give the corresponding Thom trivialization and
extreme-cell splitting after passage to $\SH(-)[1/p]$; see
Corollary~\ref{surface24splittingPositive}.

Complex Betti realization sends $S^{p,q}$ to the classical sphere
$S^p$.  Proposition~\ref{K3BettiComparison} compares the preceding
decomposition with the classical topology of the underlying K3
four-manifold and gives
\[
  \Sigma^\infty X(\C)_+
  \simeq S^0\vee\bigvee_{i=1}^{22}S^2\vee S^4,
  \qquad
  \operatorname{Re}_{\mathrm B}(F_X)
  \simeq\bigvee_{i=1}^{22}S^2.
\]
Indeed, the classical splitting already occurs after one reduced
suspension.  Motivically, however, $F_X$ is not a wedge of twenty-two
copies of $S^{2,1}$; see Remark~\ref{K3MiddleNotTate}.

The exact-order assertion in Theorem~\ref{thm:intro-k3} requires an
additional Betti-realization argument.  Applying complex Betti
realization to the point class and restricting the resulting stable
spherical fibration to the top sphere detects the classical
image-of-$J$ generator in the third stable stem.

The mechanism behind Theorem~\ref{thm:intro-k3} is a uniform
top-filtration torsion theorem for the motivic $J$-homomorphism.  For
$d>0$, let $v_p$ denote the $p$-adic valuation and set

\begin{equation}\label{eq:intro-alpha-d}
  \varepsilon_2(d)=
  \begin{cases}
    1,&d\text{ odd},\\
    v_2(d)+2,&d\text{ even},
  \end{cases}
  \qquad
  \alpha_d=
  2^{\varepsilon_2(d)}
  \prod_{\substack{p\text{ odd prime}\\p-1\mid d}}
  p^{v_p(d)+1}.
\end{equation}
Thus $\alpha_d=2$ for $d$ odd, while $\alpha_2=24$.

\enlargethispage{\baselineskip}
\begin{theorem}\label{thm:intro-j}
Let $k$ be a perfect field of characteristic zero, and let $S$ be a
smooth quasi-projective $k$-scheme of pure dimension $d>0$.  Let
$F^dK_0(S)$ denote the top step of the codimension-of-support
filtration.  Then
\[
  \alpha_d\,J_S(\xi)=0
  \quad\text{in }\pi_0\Pic(\SH(S))
\]
for every $\xi\in F^dK_0(S)$.
\end{theorem}

If $k$ has characteristic $p>0$, the same argument gives
\[
  \alpha_dJ_S(\xi)=0
  \quad\text{in }\pi_0\Pic(\SH(S)[1/p])
\]
for every $\xi\in F^dK_0(S)$.  We record this variant in
Corollary~\ref{positiveCharacteristicJ}; the integral
characteristic-zero theorem remains the principal result.

For a closed point $i:x\hookrightarrow S$, Adams--Riemann--Roch gives
$\psi^m(i_*1)=m^di_*1$.  The motivic Adams theorem then supplies an
annihilator $(m^d-1)m^{N_m}$ for every $m\ge2$, with $N_m$ unspecified.
For a prime $p$, minimizing the $p$-adic valuation over the integers
$m$ prime to $p$ gives $\varepsilon_2(d)$ when $p=2$, gives
$v_p(d)+1$ when $p$ is odd and $p-1\mid d$, and gives zero otherwise.
Combining finitely many of these annihilators yields $\alpha_d$.
Point classes generate $F^dK_0(S)$, so the same bound holds on the
entire top step.  In particular, $\alpha_2=24$.

The new point in Theorem~\ref{thm:intro-j} is to combine two known
inputs and use all Adams operations to eliminate the unspecified
powers, first for the class of a point and then for the entire
topological filtration step.  The resulting bound is an optimal formal
integral consequence of the established motivic Adams relations when
the auxiliary exponents are unspecified, and the surface value $24$ is
the arithmetic input for the geometric theorem above.

Indeed, if $X$ is a smooth projective surface with a rational point $o$ and
$c_1(T_X)=0$, $c_2(T_X)=24[o]$, then
\[
  [T_X]-[\mathcal O_X^{\oplus2}]=-24[\mathcal O_o]
  \quad\text{in }K_0(X).
\]
Theorem~\ref{thm:intro-j} therefore kills the reduced tangent class
under $J_X$.  Motivic duality then gives the Thom trivialization and
the extreme-cell splitting in Theorem~\ref{thm:intro-k3}.

We also give an exact Chern-class criterion for another source of
tangent Thom trivializations.  This criterion is sufficient but not
necessary, as the K3 application shows.

\begin{theorem}
\label{surface}
Let $k$ be a perfect field of characteristic different from $2$, and
let $X$ be a smooth projective connected surface over $k$ with a
$k$-rational point.  Assume
\[
  \CH^2(X)[2]=0,
  \qquad 2c_2(T_X)=c_1(T_X)^2\quad\text{in }\CH^2(X).
\]
Then $X_+\simeq S^{0,0}\vee F\vee S^{4,2}$ in $\SH(k)$ for some $F$
if and only if $c_1(T_X)$ is divisible by $2$ in $\CH^1(X)$.
\end{theorem}

The torsion hypothesis is needed because $c_1(T_X)=2\ell$ and the
displayed identity initially give only
$2(c_2(T_X)-2\ell^2)=0$.  In the converse direction, motivic reduced
powers show that an extreme-cell splitting forces $c_1(T_X)$ and
$c_2(T_X)$ to vanish modulo $2$, and forces
$c_1(T_X)^2-2c_2(T_X)$ to vanish modulo $3$ in characteristic
different from $3$.  As a concrete application, the split
Hirzebruch surface $\mathbb F_n$ has such a decomposition precisely
when $n$ is even.

The extreme-cell problem also has a quadratic refinement in
Milnor--Witt motives.  An $SL^c$-orientation provides the appropriate
replacement for a stable trivialization of the tangent bundle.  Recall
that R\"ondigs proved
that a pointed
smooth projective curve has a stable decomposition
$C_+\simeq S^{0,0}\vee\mathbb J(C)\vee S^{2,1}$ exactly when $T_C$
has a square root~\cite{10.1093/qmath/hap005}.  In arbitrary positive
dimension the corresponding orientation is an $SL^c$-structure on
$T_X$, namely a chosen square presentation of its determinant.

\begin{theorem}\label{thm:intro-mw}
Let $k$ be an infinite perfect field of characteristic different from
$2$, and let $X$ be a connected smooth projective $k$-scheme of
dimension $d>0$.  Fix a rational point $e\in X(k)$ and an
$SL^c$-orientation
\[
  \lambda:\det(T_X)\xrightarrow{\ \simeq\ }L^{\otimes2}.
\]
For every $\lambda$-compatible stable presentation $\sigma$ of $T_X$
in the sense of Definition~\ref{def:stable-presentation}, there is an
object $\widetilde M_{\mathrm{mid}}(X,e,\lambda,\sigma)$ and an
isomorphism
\[
  \widetilde M(X)\simeq
  \widetilde\Z\oplus\widetilde M_{\mathrm{mid}}(X,e,\lambda,\sigma)
  \oplus\widetilde\Z(d)[2d]
\]
in $\DMW(k)$.
\end{theorem}

Such stable presentations exist by the Jouanolou trick.  The
normalized Thom-isomorphism theorem used here does not, by itself,
identify the self-dualities arising from different presentations, so
the datum $\sigma$ is retained in the notation.  No finite
cohomological-dimension hypothesis is required.  The result applies,
in particular, to pointed smooth projective varieties with trivial
canonical bundle and
to curves with a theta characteristic.  We also compare it directly
with stable splittings.  The strong symmetric monoidal functors
\[
  \SH(k)\longrightarrow\DMW(k)\longrightarrow\DM(k)
\]
carry an extreme-cell splitting forward, and the resulting idempotents
become the two standard extreme Chow projectors after passage to
ordinary motives.  In $\DMW(k)$, the realized top projector is
conjugate to the one defined by an independently chosen
$SL^c$-orientation and compatible stable presentation, and the two
agree under the compatibility condition of
Proposition~\ref{stableMWcompat}.  The functors are not generally
full, so this comparison does not run in reverse.

This refinement also admits a cycle-level treatment.  Normalized
Chow--Witt Gysin maps construct the two extreme quadratic projectors
and satisfy the required idempotence and orthogonality relations.  An
additive lift of a Chow projector need not remain idempotent, so for a
general finite system we prove that, if
$\operatorname{cd}_2(k)=c<\infty$, the kernel of the comparison
homomorphism on correspondence rings is nilpotent of exponent at most
$c+d+1$.  Standard simultaneous idempotent lifting then applies.

We conclude the introduction with an outline of the paper.  In
Section~\ref{sectionsurface}, we prove the top-filtration torsion
theorem, recall the duality criterion, and obtain the surface
splitting theorem.  In Section~\ref{sectionchern}, we develop the
complementary Chern-class criteria and reduced-power obstructions.  In
Section~\ref{sectionK}, we construct the $SL^c$-oriented
Milnor--Witt realization, while Section~\ref{sec:cycles} gives the
cycle-level construction of the extreme Chow--Witt projectors by
normalized Gysin maps.  Appendix~\ref{appendix:nilpotence} contains
the general nilpotent-kernel theorem.  As a supplementary result, the
final appendix records an integral Chow-motive decomposition for
principally polarized abelian surfaces.  Under the comparison functor
$\DMW(k)\to\DM(k)$, these projectors give a concrete target system for
the Chow--Witt lifting problem, with liftability governed by the image
condition in Corollary~\ref{simultaneousprojectors}.

\subsection*{Acknowledgments}
The author thanks Aravind Asok for valuable comments and suggestions
on an earlier version of this paper, and Tianle Liu for helpful
discussions.  The author also thanks Fangzhou Jin for his hospitality
during a visit to Tongji University.

\section{Top-filtration motivic
\texorpdfstring{$J$}{J}-torsion and extreme cells}
\label{sectionsurface}
In this section, we first recall the duality and connectivity argument
for splitting off the extreme cells, and then study the motivic
$J$-homomorphism on the topological filtration.  Throughout the
section, $k$ denotes a perfect field.  When discussing extreme cells,
$X$ is smooth, projective, and connected; in
Subsection~\ref{subsec:codim-two}, $S$ is smooth and quasi-projective
of arbitrary positive pure dimension.  We
write $\one$ for the tensor unit of $\SH(k)$ and $X_+$ for the
$\T$-suspension spectrum of $X$ with a disjoint base point.  We use the
standard bigraded spheres $S^{p,q}$ and the Thom object $\Th_X(V)$ of a
virtual vector bundle $V$ on $X$.
Background may be found in~\cite{MV99,Jardine2000,Dundas2003,10.1093/qmath/hap005}.

The construction of the splitting uses stable motivic connectivity in
the following form~\cite[Theorem 3.1]{10.1093/qmath/hap005}.

\begin{theorem}\label{Conn}
For $p,q\in\Z$ with $p>q$,
\[
  \Hom_{\SH(k)}(S^{0,0},S^{p,q})=0.
\]
\end{theorem}

We next recall motivic Atiyah duality.  For a strongly dualizable
object $E$, write
$\mathcal D(E)=\underline{\Hom}(E,\one)$.  Motivic Atiyah duality has
the following form~\cite{HU2005609,AST_2007__314__R1_0}.

\begin{theorem}\label{Dual}
Let $p:X\to\Spec(k)$ be smooth and projective.  There is a canonical
isomorphism
\[
  \mathcal D(X_+)\simeq p_\sharp\Th_X(-T_X)
\]
in $\SH(k)$.
\end{theorem}

Since $\SH(k)$ is idempotent complete, every retract splits.  A
section $\one\to X_+$ of the structure map will be called a rational
point up to stable homotopy; an ordinary $k$-rational point supplies
such a section.

The following criterion is the abstract form of the
duality--connectivity argument used for abelian varieties
in~\cite[Theorem 3.3]{liu2024splittingabelianvarietiesmotivic}.

\begin{theorem}\label{Shift1}
Let $X$ have dimension $d>0$, and suppose that $\one$ is a retract of
$X_+$.  If
\[
  \mathcal D(X_+)\simeq S^{-2d,-d}\wedge X_+,
\]
then there is an object $F\in\SH(k)$ and an isomorphism
\[
  X_+\simeq\one\vee F\vee S^{2d,d}.
\]
\end{theorem}
\begin{proof}
Choose maps $i_0:\one\to X_+$ and $r_0:X_+\to\one$ with
$r_0i_0=\id$.  Idempotent completeness gives
$X_+\simeq\one\vee N$.  Dualizing this retract and using the assumed
self-duality produces maps
\[
  i_{2d}:S^{2d,d}\longrightarrow X_+,
  \qquad r_{2d}:X_+\longrightarrow S^{2d,d},
  \qquad r_{2d}i_{2d}=\id.
\]
By Theorem~\ref{Conn}, $r_{2d}i_0=0$.  Hence $r_{2d}$ factors through
$N$, and the projection of $i_{2d}$ to $N$ still has left inverse
$r_{2d}|_N$.  Thus $S^{2d,d}$ is a retract of $N$.  Splitting the
corresponding idempotent gives $N\simeq F\vee S^{2d,d}$.
\end{proof}

The Thom construction defines a homomorphism
\[
  J_X=\pi_0\Th:K_0(X)\longrightarrow
  \pi_0\Pic(\SH(X)).
\]
R\"ondigs proves that the Thom objects of a vector bundle and its dual
are isomorphic in $\SH(X)$~\cite[Propositions 2.1--2.2]
{10.1093/qmath/hap005}.  Thus the same holds for virtual bundles.

\begin{proposition}\label{Jcriterion}
Let $X$ be smooth, projective, and connected of dimension $d>0$, and
assume that $X(k)\ne\varnothing$.  If
\[
  J_X\bigl([T_X]-[\mathcal O_X^{\oplus d}]\bigr)=0,
\]
then there is an object $F\in\SH(k)$ and an isomorphism
\[
  X_+\simeq S^{0,0}\vee F\vee S^{2d,d}.
\]
The same implication holds in $\SH(k)[1/e]$ if the displayed
$J$-class vanishes in $\pi_0\Pic(\SH(X)[1/e])$, where $e$ is the
exponential characteristic of $k$.
\end{proposition}

\begin{proof}
The equality under $J_X$ is equivalent to a Thom isomorphism
\[
  \Th_X(T_X)\simeq\Th_X(\mathcal O_X^{\oplus d})
  \simeq S_X^{2d,d}.
\]
Theorem~\ref{Dual} and the projection formula give
$\mathcal D(X_+)\simeq S^{-2d,-d}\wedge X_+$.  The rational point
splits the structure map, so Theorem~\ref{Shift1} applies.
After inverting $e$, motivic duality and the projection formula commute
with the symmetric monoidal localization.  The connectivity assertion in
Theorem~\ref{Conn} still holds after localization, so the same proof
gives the localized assertion.
\end{proof}

\begin{corollary}\label{Kcriterion}
Let $X$ be smooth, projective, and connected of dimension $d>0$, and
assume that $X(k)\ne\varnothing$.  If
\[
  [T_X]-[\mathcal O_X^{\oplus d}]
  \in\im\bigl(1-\vee:K_0(X)\longrightarrow K_0(X)\bigr),
\]
then $X_+\simeq S^{0,0}\vee F\vee S^{2d,d}$ in $\SH(k)$ for some
$F$.
\end{corollary}

\begin{proof}
The dual-bundle isomorphism implies
$J_X(V-V^\vee)=0$ for every $V\in K_0(X)$.  The hypothesis therefore
implies
$J_X([T_X]-[\mathcal O_X^{\oplus d}])=0$, and
Proposition~\ref{Jcriterion} applies.
\end{proof}

\subsection{Top-filtration torsion}
\label{subsec:codim-two}

We now study the motivic $J$-homomorphism on the top step of the
codimension-of-support filtration.  The principal statements are
integral and assume that $k$ has characteristic zero.  For a regular
scheme $S$, we write
$F^qK_0(S)$ for the codimension-of-support filtration: under
$K_0(S)\simeq G_0(S)$, it is generated by coherent sheaves whose
support has codimension at least $q$.  The positive-characteristic
variant is recorded separately in
Corollary~\ref{positiveCharacteristicJ}.

\begin{lemma}\label{adamsArithmetic}
For a prime $p$ and an integer $d>0$, put
\[
  \mu_p(d)=\min_{\substack{m\ge2\\p\nmid m}}v_p(m^d-1).
\]
Then
\[
  \mu_2(d)=\varepsilon_2(d)
\]
and, for an odd prime $p$,
\[
  \mu_p(d)=
  \begin{cases}
    v_p(d)+1,&p-1\mid d,\\
    0,&p-1\nmid d.
  \end{cases}
\]
\end{lemma}

\begin{proof}
Suppose first that $p$ is odd.  If $p-1\nmid d$, the cyclic group
$\mathbb F_p^\times$ contains an element whose $d$th power is not one;
any integral lift gives $\mu_p(d)=0$.  If $p-1\mid d$, Fermat's theorem
shows that $v_p(m^d-1)\ge1$ whenever $p\nmid m$.  Choose $m$ whose
residue class is a primitive root modulo $p^2$.  Then
$v_p(m^{p-1}-1)=1$, and the lifting-the-exponent formula gives
\[
  v_p(m^d-1)=1+v_p\!\left(\frac d{p-1}\right)=v_p(d)+1.
\]

For $p=2$ and $d$ odd, the choice $m=3$ gives
$v_2(3^d-1)=1$.  If $d$ is even, the same formula gives, for every odd
$m$,
\[
  v_2(m^d-1)=v_2(m-1)+v_2(m+1)+v_2(d)-1\ge v_2(d)+2,
\]
and equality holds for $m=3$.
\end{proof}

\begin{theorem}\label{topPointJtorsion}
Let $k$ be a perfect field of characteristic zero, let $S$ be a smooth
quasi-projective $k$-scheme of pure dimension $d>0$, and let
$i:x\hookrightarrow S$ be a closed point.  Put
\[
  u_x=i_*1=[\mathcal O_x]\in K_0(S).
\]
Then
\[
  \alpha_d\,J_S(u_x)=0
  \quad\text{in }\pi_0\Pic(\SH(S)).
\]
\end{theorem}

\begin{proof}
Since $k$ is perfect, the closed point $x$ is smooth over $k$, and the
immersion $i$ is regular of codimension $d$.  Adams--Riemann--Roch
without denominators gives, for every integer $m\ge1$,
\[
  \psi^m(i_*1)=i_*\bigl(\theta^m(N_i^\vee)\bigr);
\]
see~\cite[Proposition 4.12]{GilletSoule1987} and
\cite[Corollary 3.4]{NavarroNavarroAdamsRR}.  Since $x$ is the spectrum
of a field, $K_0(x)\simeq\Z$, and the conormal bundle $N_i^\vee$ has
rank $d$.  Its cannibalistic class is therefore $m^d$, so
\begin{equation}\label{eq:adams-point}
  \psi^m(u_x)=m^du_x.
\end{equation}

The motivic Adams theorem states that, for every virtual vector bundle
$V$ on $S$ and every integer $m$, there is an integer $N(V,m)\ge0$
such that
\begin{equation}\label{eq:motivic-adams-general}
  m^{N(V,m)}J_S\bigl(V-\psi^m(V)\bigr)=0
\end{equation}
in $\pi_0\Pic(\SH(S))$; see~\cite[Theorem 6.5 and (1.1)]
{ananyevskiy2023motivicadamsconjecture}.  Apply
\eqref{eq:motivic-adams-general} to $V=u_x$.  This is legitimate
because the regular immersion $i$ makes $i_*\mathcal O_x$ a perfect
complex on the regular scheme $S$; under $K_0(S)\simeq G_0(S)$, its
class $u_x$ is therefore a virtual vector bundle.

The exponent may depend on both $V$ and $m$; with $V=u_x$ fixed, write
$N_m=N(u_x,m)$.  Equations~\eqref{eq:adams-point}
and~\eqref{eq:motivic-adams-general} give
\begin{equation}\label{eq:all-adams-annihilators}
  (m^d-1)m^{N_m}J_S(u_x)=0
  \qquad(m\ge2).
\end{equation}
Let $P$ be the finite set of prime divisors of
$(2^d-1)2^{N_2}$.  For each $p\in P$, choose, by
Lemma~\ref{adamsArithmetic}, an integer $m_p\ge2$ prime to $p$ such
that
\[
  v_p(m_p^d-1)=\mu_p(d).
\]
Let $g$ be the greatest common divisor of $(2^d-1)2^{N_2}$ and the
finitely many integers $(m_p^d-1)m_p^{N_{m_p}}$ for $p\in P$.
B\'ezout's identity and~\eqref{eq:all-adams-annihilators} give
$gJ_S(u_x)=0$.  Since $g$ divides the first coefficient, it has no
prime divisor outside $P$; for $p\in P$, the choice $p\nmid m_p$ gives
\[
  v_p(g)\le v_p(m_p^d-1)=\mu_p(d).
\]
Lemma~\ref{adamsArithmetic} and~\eqref{eq:intro-alpha-d} therefore
show that $g\mid\alpha_d$.  Hence $\alpha_dJ_S(u_x)=0$.
\end{proof}

\begin{remark}
The integer $\alpha_d$ is the strongest dimension-only bound obtainable
from the full family of relations~\eqref{eq:all-adams-annihilators}
without further information on the exponents $N_m$.  Indeed, for each
prime $p$, the Adams operations with $p\nmid m$ force precisely the
minimum $\mu_p(d)$ of Lemma~\ref{adamsArithmetic}.  Conversely, one may
choose the unspecified exponents for the integers divisible by $p$ so
large that every coefficient in~\eqref{eq:all-adams-annihilators} is
divisible by $p^{\mu_p(d)}$.  Given any proper divisor of $\alpha_d$,
choose a prime at which its valuation is smaller.  The preceding
choice shows that this divisor is not a formal consequence of the
relations for all possible values of the unspecified exponents.
\end{remark}

\begin{corollary}\label{topFiltrationJtorsion}
Let $k$ be a perfect field of characteristic zero and let $S$ be a
smooth quasi-projective $k$-scheme of pure dimension $d>0$.  Then
\[
  \alpha_d\,J_S(\xi)=0
  \quad\text{in }\pi_0\Pic(\SH(S))
\]
for every $\xi\in F^dK_0(S)$.
\end{corollary}

\begin{proof}
Since $S$ is regular, $K_0(S)=G_0(S)$, and $F^dG_0(S)$ is generated by
the classes of coherent sheaves with zero-dimensional support.  Every
such sheaf has a finite filtration whose factors are residue fields of
closed points.  Theorem~\ref{topPointJtorsion} therefore applies
additively.
\end{proof}

\begin{corollary}\label{positiveCharacteristicJ}
Let $k$ be a perfect field of characteristic $p>0$, and let $S$ be a
smooth quasi-projective $k$-scheme of pure dimension $d>0$.  Then
\[
  \alpha_dJ_S(\xi)=0
  \quad\text{in }\pi_0\Pic(\SH(S)[1/p])
\]
for every $\xi\in F^dK_0(S)$.
\end{corollary}

\begin{proof}
Theorem~6.5 of~\cite{ananyevskiy2023motivicadamsconjecture}, with
exponential characteristic $e=p$, is the required localized form of
the motivic Adams theorem.  Repeating the proof of
Theorem~\ref{topPointJtorsion} first gives
$\alpha_dJ_S([\mathcal O_x])=0$ in
$\pi_0\Pic(\SH(S)[1/p])$ for every closed point $x$.  The generation argument of
Corollary~\ref{topFiltrationJtorsion} then gives the assertion for
$F^dK_0(S)$.
\end{proof}

\begin{corollary}\label{pointJtorsion}
Let $k$ be a perfect field of characteristic zero, let $S$ be a smooth
quasi-projective surface over $k$, and let $i:x\hookrightarrow S$ be a
closed point.  Then
\[
  24\,J_S(i_*1)=0
  \quad\text{in }\pi_0\Pic(\SH(S)).
\]
\end{corollary}

\begin{proof}
This is Theorem~\ref{topPointJtorsion} with $d=2$, for which
$\alpha_2=24$.
\end{proof}

\begin{corollary}\label{surfaceF2Jtorsion}
Let $k$ be a perfect field of characteristic zero and let $S$ be a
smooth quasi-projective surface over $k$.  Then
\[
  24\,J_S(\xi)=0
  \quad\text{in }\pi_0\Pic(\SH(S))
\]
for every $\xi\in F^2K_0(S)$.
\end{corollary}

\begin{proof}
This is Corollary~\ref{topFiltrationJtorsion} with $d=2$.
\end{proof}

\begin{remark}\label{surfaceJscale}
Over a field of characteristic zero, the second motivic Hopf map
$\nu$ generates $\pi_{3,2}\one\simeq\Z/24$~\cite[Theorem 5.5 and
Remark 5.8]{RondigsSpitzweckOstvaer2019}.  Thus the surface exponent
occurs already at the sphere level.  Proposition~\ref{complexSurfaceJSharp}
will show that the exponent is attained by every point class on a
complex surface admitting an extreme-cell splitting.
\end{remark}

\begin{remark}\label{positive-characteristic-J-warning}
For surfaces in characteristic $p>0$, the annihilator in
Corollary~\ref{positiveCharacteristicJ} is uniformly $24$ after
passage to the localized category $\SH(S)[1/p]$.  The sharpness
argument of Proposition~\ref{complexSurfaceJSharp} uses complex Betti
realization and applies in characteristic zero.
\end{remark}

\subsection{Surfaces with a distinguished zero-cycle and complex K3
surfaces}\label{sectionK3}

We now apply the integral surface torsion bound.  The following theorem
isolates the required geometric input; the surface need not be a K3
surface.

\begin{theorem}\label{surface24splitting}
Let $k$ be a perfect field of characteristic zero, let $X$ be a smooth
projective connected surface over $k$, and let $o\in X(k)$.  Suppose
that
\[
  c_1(T_X)=0,
  \qquad c_2(T_X)=24[o]
  \quad\text{in }\CH^*(X).
\]
Then
\[
  \Th_X\bigl(T_X-\mathcal O_X^{\oplus2}\bigr)\simeq\one_X
\]
in $\SH(X)$.  Consequently, for some $F_X\in\SH(k)$ there is a
decomposition
\[
  X_+\simeq S^{0,0}\vee F_X\vee S^{4,2}
\]
in $\SH(k)$.
\end{theorem}

\begin{proof}
Put $u=[\mathcal O_o]\in K_0(X)$.  For a smooth surface, the
topological filtration satisfies $F^2K_0(X)\simeq\CH^2(X)$ and
$F^3K_0(X)=0$; on a rank-zero class with trivial first Chern class,
this identification is given by $-c_2$~\cite[Example 15.3.6]{fulton}.
Consequently,
\begin{equation}\label{eq:surface-24-k0}
  [T_X]-[\mathcal O_X^{\oplus2}]=-24u
  \qquad\text{in }K_0(X).
\end{equation}
Corollary~\ref{pointJtorsion} gives $24\,J_X(u)=0$.
Equation~\eqref{eq:surface-24-k0} therefore implies
\[
  J_X\bigl([T_X]-[\mathcal O_X^{\oplus2}]\bigr)=0
\]
in $\pi_0\Pic(\SH(X))$, which is equivalent to the asserted Thom
isomorphism.  Proposition~\ref{Jcriterion}, with $o$ as rational
point, gives the decomposition of $X_+$.
\end{proof}

\begin{corollary}\label{surface24splittingPositive}
Let $k$ be a perfect field of characteristic $p>0$, let $X$ be a
smooth projective connected surface over $k$, and let $o\in X(k)$.
If
\[
  c_1(T_X)=0,
  \qquad c_2(T_X)=24[o]
  \quad\text{in }\CH^*(X),
\]
then the conclusions of Theorem~\ref{surface24splitting} hold in
$\SH(X)[1/p]$ and $\SH(k)[1/p]$, respectively.
\end{corollary}

\begin{proof}
Equation~\eqref{eq:surface-24-k0} remains valid.  By
Corollary~\ref{positiveCharacteristicJ}, one has
\[
  24J_X(u)=0
  \quad\text{in }\pi_0\Pic(\SH(X)[1/p]).
\]
The localized form of Proposition~\ref{Jcriterion} applies.
\end{proof}

\begin{corollary}\label{K3SHsplitting}
Let $X$ be a smooth projective K3 surface over $\C$.  Then
\[
  \Th_X\bigl(T_X-\mathcal O_X^{\oplus2}\bigr)\simeq\one_X
\]
in $\SH(X)$.  Consequently, for some $F_X\in\SH(\C)$ there is a
decomposition
\[
  X_+\simeq S^{0,0}\vee F_X\vee S^{4,2}
\]
in $\SH(\C)$.
\end{corollary}

\begin{proof}
Choose a Beauville--Voisin point $o\in X(\C)$, represented by a point
on a rational curve~\cite[Theorem 1]{Beauville2001ONTC}.  The
Beauville--Voisin identity~\cite[Theorem 1(c)]{Beauville2001ONTC}
gives
\[
  c_1(T_X)=0,
  \qquad c_2(T_X)=24[o]\in\CH^2(X).
\]
Theorem~\ref{surface24splitting} applies in characteristic zero.
\end{proof}

\subsection{Betti realization, sharpness, and the classical stable splitting}
\label{subsec:K3-classical-comparison}

We next compare the motivic splitting with the classical stable
homotopy type and use this comparison to prove that the exponent $24$
is sharp.

Let
\[
  \operatorname{Re}_{\mathrm B}:\SH(\C)\longrightarrow
  \SH_{\mathrm{top}}
\]
denote complex Betti realization, where $\SH_{\mathrm{top}}$ is the
classical stable homotopy category.  This is an exact symmetric
monoidal functor satisfying
\[
  \operatorname{Re}_{\mathrm B}(X_+)
  \simeq\Sigma^\infty X(\C)_+,
  \qquad
  \operatorname{Re}_{\mathrm B}(S^{p,q})\simeq S^p;
\]
see~\cite[Section 6]{Levine2014Comparison}.

For a finite CW-complex $M$, the stable complex $J$-homomorphism sends
a virtual complex vector bundle of rank zero to its stable spherical
fibration.  On infinite loop spaces, it is the map
\[
  BU\longrightarrow BGL_1(\mathbb S)
\]
obtained by taking fibrewise one-point compactifications of the
underlying real bundles.  Since $BGL_1(\mathbb S)$ is an infinite loop
space, this construction represents a reduced generalized cohomology
class and is contravariantly functorial for stable maps.  Relative
Betti realization carries algebraic Thom objects over $X$ to the
fibrewise one-point compactifications of their analytifications over
$X(\C)$.  In particular, it carries the motivic $J$-class of a
virtual bundle to its topological $J$-class; the resulting comparison
with complex Betti realization is the commutative diagram of
\cite[Remark 1.2]{ananyevskiy2023motivicadamsconjecture}.

\begin{proposition}\label{complexSurfaceJSharp}
Let $k$ be a field of characteristic zero equipped with an embedding
$\sigma:k\hookrightarrow\C$, and let $X$ be a smooth projective
connected surface over $k$.  Put $M=X_\sigma(\C)$ and suppose that
there are stable maps
\[
  S^4\xrightarrow{s}\Sigma^\infty M_+
  \xrightarrow{r}S^4,
  \qquad rs=\id_{S^4}.
\]
Then, for every $x\in X(k)$, the class
\[
  J_X([\mathcal O_x])\in\pi_0\Pic(\SH(X))
\]
has exact order $24$.  Moreover, after base change along $\sigma$, the
same conclusion holds for every point of $X_\sigma(\C)$.  The
hypothesis holds in particular if $X_+$ admits an extreme-cell
splitting in $\SH(k)$.
\end{proposition}

\begin{proof}
Since $H^4(M,\Z)\simeq\Z$ and $rs=\id$, the map $s$ induces
multiplication by $\pm1$ in degree-four integral cohomology.  Let
$u_x^{\mathrm{top}}\in K^0(M)$ be the analytification of
$[\mathcal O_x]$ after base change along $\sigma$.

Grothendieck--Riemann--Roch for the regular immersion
$x\hookrightarrow X$, followed by the cycle-class map, gives
\[
  \operatorname{ch}(u_x^{\mathrm{top}})
  =\operatorname{PD}[x]\in H^4(M,\Z)\subset H^{\mathrm{even}}(M,\Q).
\]
See~\cite[Theorem 15.2]{fulton} for the Riemann--Roch formula used
here. It follows, by stable functoriality of complex $K$-theory, that
\[
  s^*u_x^{\mathrm{top}}=\pm b^2
  \quad\text{in }\widetilde K^0(S^4)\simeq\Z,
\]
where $b^2$ is the degree-four complex Bott generator.  Indeed, the
Chern character of $b^2$ is a generator of $H^4(S^4,\Z)$.

We make the realification step explicit.  Let
$r:\widetilde K^0(S^4)\to\widetilde{KO}^0(S^4)$ and
$c:\widetilde{KO}^0(S^4)\to\widetilde K^0(S^4)$ be realification and
complexification.  Bott periodicity identifies both groups with
$\mathbb Z$; if $q$ is the quaternionic Bott generator, then
$c(q)=2b^2$.  Since $c\,r(b^2)=b^2+\overline{b^2}=2b^2$, it follows
that $r(b^2)=q$.  The complex stable $J$-homomorphism is the composite
of realification with the real stable $J$-homomorphism.  Hence it sends
$b^2$ to a generator of
\[
  \pi_4BGL_1(\mathbb S)=\pi_3^S\simeq\Z/24;
\]
see~\cite[Theorem 1.5]{Adams1966JIV}.  Naturality of the topological
$J$-homomorphism and the preceding comparison therefore show that the
restriction along $s$ of the Betti realization of
$J_X([\mathcal O_x])$ is the $J$-image of $\pm b^2$ and has exact
order $24$.
Hence $J_X([\mathcal O_x])$ has order divisible by $24$.  On the other
hand, Corollary~\ref{surfaceF2Jtorsion} gives
$24J_X([\mathcal O_x])=0$.  Its order is therefore exactly $24$.
The same proof applies to every point of $X_\sigma(\C)$.  Finally, an
extreme-cell splitting over $k$ yields the displayed retract after
base change and Betti realization.
\end{proof}

Together with Theorem~\ref{surface24splitting}, the proposition proves
the exact-order assertion in Theorem~\ref{thm:intro-k3}.

\begin{corollary}\label{K3PointJSharp}
Let $X$ be a smooth projective K3 surface over $\C$.  Then
$J_X([\mathcal O_x])$ has exact order $24$ for every $x\in X(\C)$.
In particular, the exponent in Corollary~\ref{surfaceF2Jtorsion} is
optimal even after restricting to point classes on complex K3
surfaces.
\end{corollary}

\begin{proof}
Combine Corollary~\ref{K3SHsplitting} with
Proposition~\ref{complexSurfaceJSharp}.
\end{proof}

The corollary determines the order of each point class separately;
equality among the classes $J_X([\mathcal O_x])$ as $x$ varies is a
distinct question.

\begin{proposition}\label{K3BettiComparison}
Let $X$ be a smooth projective K3 surface over $\C$, and let
\[
  X_+\simeq S^{0,0}\vee F_X\vee S^{4,2}
\]
be a splitting supplied by Corollary~\ref{K3SHsplitting}.  Its complex
Betti realization is noncanonically equivalent to the classical
stable splitting
\begin{equation}\label{eq:classical-k3-splitting}
  \Sigma^\infty X(\C)_+
  \simeq
  S^0\vee\bigvee_{i=1}^{22}S^2\vee S^4.
\end{equation}
More precisely, the equivalence in~\eqref{eq:classical-k3-splitting}
may be chosen so that the realized extreme projectors agree with the
standard bottom- and top-sphere projectors.
In particular,
\[
  \operatorname{Re}_{\mathrm B}(F_X)
  \simeq\bigvee_{i=1}^{22}S^2.
\]
More strongly, after choosing a base point on $X(\C)$, there is an
equivalence of pointed spaces
\begin{equation}\label{eq:once-suspended-k3}
  \Sigma X(\C)\simeq\bigvee_{i=1}^{22}S^3\vee S^5.
\end{equation}
\end{proposition}

\begin{proof}
Put $M=X(\C)$.  With its cup-product pairing, the middle cohomology
lattice is
\[
  \bigl(H^2(M,\Z),q_M\bigr)
  \simeq E_8(-1)^{\oplus2}\oplus U^{\oplus3},
\]
where $U$ is the hyperbolic plane.  In particular, it is even and has
rank $22$; see~\cite[Proposition 1.3.5 and
Corollary 7.1.4]{HuybrechtsK3}.  A complex K3 surface is simply
connected; see~\cite[Remark 1.3.6(i)]{HuybrechtsK3}.  The classical
description of simply connected four-dimensional Poincar\'e complexes
gives a based
homotopy equivalence
\begin{equation}\label{eq:k3-classical-cw-model}
  M\simeq W\cup_\alpha e^4,
  \qquad
  W=\bigvee_{i=1}^{22}S_i^2,
\end{equation}
for a map $\alpha:S^3\to W$; see~\cite{Whitehead1949}.  This is a
consequence also of~\cite[Proposition 4C.3]{MR1867354}.  We use this
only as a based CW model; no handle decomposition is involved.

Let $\iota_i:S_i^2\to W$ be the wedge inclusions, let
$\eta:S^3\to S^2$ be the Hopf map, and put
$\eta_i=\iota_i\eta$.  Choose the cellular basis
$x_1,\ldots,x_{22}$ of $H^2(M,\Z)$ and an orientation class
$u\in H^4(M,\Z)$, and write
\[
  x_i\smile x_j=q_{ij}u.
\]
After fixing compatible orientations, the Hilton
decomposition~\cite[Theorem A, p.~155]{Hilton1955} and the Hopf-invariant
description of cup products~\cite[Proposition 4C.3]{MR1867354}
identify the attaching class as
\begin{equation}\label{eq:k3-attaching-map}
  [\alpha]
  =\sum_{i=1}^{22}q_{ii}\eta_i
   +\sum_{1\leq i<j\leq22}q_{ij}[\iota_i,\iota_j]
  \quad\text{in }\pi_3(W).
\end{equation}
Suspension annihilates Whitehead products, whereas $\Sigma\eta$
generates $\pi_4(S^3)\simeq\Z/2$;
see~\cite[Corollary 4J.4]{MR1867354}.  Since the intersection form is
even, each $q_{ii}$ is even.  Hence $\Sigma\alpha$ is null-homotopic.
Suspending the cofibre sequence defined
by~\eqref{eq:k3-classical-cw-model} gives
\[
  \Sigma M\simeq\Sigma W\vee S^5
  \simeq\bigvee_{i=1}^{22}S^3\vee S^5,
\]
which proves~\eqref{eq:once-suspended-k3} and, after stabilization and
adjoining the bottom sphere,~\eqref{eq:classical-k3-splitting}.

Applying $\operatorname{Re}_{\mathrm B}$ to the motivic decomposition
gives
\[
  \Sigma^\infty M_+
  \simeq S^0\vee\operatorname{Re}_{\mathrm B}(F_X)\vee S^4.
\]
The bottom projector is carried to the projector defined by the base
point and the map $M\to *$, so it remains to compare the reduced
part.  Write
\[
  V=\bigvee_{i=1}^{22}S^2,
  \qquad T=S^4.
\]
The first part of the proof identifies
$\Sigma^\infty M\simeq V\vee T$.  The realized top retraction
restricts trivially to $V$, since
$[S^2,S^4]_{\SH_{\mathrm{top}}}=0$, and has degree $\pm1$ on $T$.
After composing the $T$-summand with the degree-$-1$ self-map if
necessary, its restriction to $T$ is the identity.  Its section is
then the graph of a stable map $T\to V$.  A triangular
self-equivalence of $V\vee T$ carries this graph to the standard
$T$-summand.  The complementary summand is therefore equivalent to
$V$, proving both the projector comparison and
$\operatorname{Re}_{\mathrm B}(F_X)\simeq V$.
\end{proof}

\begin{remark}\label{K3MiddleNotTate}
The attaching map $\alpha$ in~\eqref{eq:k3-attaching-map} is not
null-homotopic: its Hopf and Whitehead-product terms encode the
nontrivial K3 intersection form.  Suspension kills the latter and
reduces the former modulo $2$.  More generally, for a closed simply
connected oriented four-manifold, the stable top-cell obstruction is
the coordinate vector $(q_{ii}\bmod 2)_i$.  Via Poincar\'e duality
and the Wu formula, this vector corresponds to the second
Stiefel--Whitney class.  Thus the classical splitting of a K3 surface
is a consequence of its being spin.  By contrast, the odd
intersection form of $\mathbb P^2(\C)$ leaves the stable Hopf map
$\eta$ as its top-cell obstruction.

Betti realization forgets motivic weight and is not conservative.
Although
$\operatorname{Re}_{\mathrm B}(F_X)\simeq\bigvee_{i=1}^{22}S^2$, the
middle summand in Corollary~\ref{K3SHsplitting} is not a wedge of
twenty-two copies of $S^{2,1}$.  Such an equivalence would give
\[
  H^{2,1}(X,\Z)\simeq\Z^{22}.
\]
For a complex K3 surface, however,
$H^{2,1}(X,\Z)\simeq\Pic(X)$ and the Picard rank is at most
$h^{1,1}(X)=20$.  Thus the theorem splits only the two extreme
motivic spheres; the transcendental degree-two information remains in
the middle summand.  The surface $J$-torsion argument provides
motivic lifts of the two classical extreme sphere summands, but not of
the full classical wedge decomposition.  The suspended cellular
obstruction used here detects only parity, whereas the factor $24$
enters through the tangent/$J$-theoretic proof and is not detected by
this elementary cell calculation.
\end{remark}

\section{Chern-class criteria and reduced-power obstructions}
\label{sectionchern}

In this section, we give a second criterion for splitting off the
extreme cells and then use motivic reduced powers to obtain converse
obstructions.  The exact Chern identities below are sufficient for
$J$-triviality, but Corollary~\ref{K3SHsplitting} shows that they are
not necessary.  Throughout the section, $k$ is a perfect field and
$X$ is a smooth projective connected surface over $k$.

\begin{theorem}\label{chernsufficient}
Let $X$ be a smooth projective connected surface over $k$ with a
$k$-rational point.  Suppose that there is a line bundle $L$ on $X$
such that, with $\ell=c_1(L)$,
\begin{equation}\label{eq:surface-exact-chern}
  c_1(T_X)=2\ell,
  \qquad c_2(T_X)=2\ell^2
  \quad\text{in }\CH^*(X).
\end{equation}
Then
\[
  X_+\simeq S^{0,0}\vee F\vee S^{4,2}
\]
in $\SH(k)$ for some $F$.
\end{theorem}

\begin{proof}
Consider
\[
  \varepsilon=[T_X]-[\mathcal O_X^{\oplus2}]-[L]+[L^\vee]
  \in K_0(X).
\]
The identities~\eqref{eq:surface-exact-chern} give
\[
  \operatorname{rk}(\varepsilon)=0,
  \qquad c_1(\varepsilon)=0,
  \qquad c_2(\varepsilon)=c_2(T_X)-2\ell^2=0.
\]
For a smooth surface, rank, determinant, and the second Chern class
determine a class in $K_0(X)$; equivalently, this follows from the
topological filtration and $F^3K_0(X)=0$~\cite[Example 15.3.6]{fulton}.
Hence $\varepsilon=0$, so
\[
  [T_X]-[\mathcal O_X^{\oplus2}]=[L]-[L^\vee].
\]
Corollary~\ref{Kcriterion} now applies.
\end{proof}

\begin{example}\label{example:surface-sufficient}
An abelian surface satisfies Theorem~\ref{chernsufficient} with
$L=\mathcal O_X$, because its tangent bundle is trivial.  More
generally, let $C_1$ and $C_2$ be smooth projective curves with
rational points and suppose that $T_{C_i}\simeq L_i^{\otimes2}$.
For $X=C_1\times C_2$, take
$L=p_1^*L_1\otimes p_2^*L_2$.  Since the squares of divisor classes
pulled back from the individual curves vanish, one obtains
$c_1(T_X)=2c_1(L)$ and $c_2(T_X)=2c_1(L)^2$.
\end{example}

\begin{remark}\label{K3remark}
For a K3 surface, the first identity in
\eqref{eq:surface-exact-chern} would give $2\ell=c_1(T_X)=0$.
Since $\Pic(X)$ is torsion-free, this forces $\ell=0$, whereas the
second identity would then give $c_2(T_X)=0$, contradicting
$\deg c_2(T_X)=24$.  Its splitting instead follows from the
$J$-torsion argument of Corollary~\ref{K3SHsplitting}.
\end{remark}

We now recall the Thom--Wu formulas needed for the converse
obstructions.  A bar denotes reduction in the appropriate Chow group.

\begin{lemma}\label{thomwu}
Let $a\in K_0(X)$ and let $u_a$ be its Thom class.  If
$\operatorname{char}(k)\ne2$, then
\[
  Sq^2(u_a)=\overline{c_1(a)}u_a,
  \qquad Sq^4(u_a)=\overline{c_2(a)}u_a.
\]
If $\operatorname{char}(k)\ne3$, then
\[
  P^1(u_a)=
  \bigl(\overline{c_1(a)}^2-2\overline{c_2(a)}\bigr)u_a.
\]
\end{lemma}

\begin{proof}
For an actual vector bundle these are the motivic Thom formulas for
reduced power operations~\cite[Theorem 14.2 and Corollary
14.4]{58}; the first formula also appears in the form used here
in~\cite{10.1093/qmath/hap005}.  For completeness, let
$a=[E]-[F]$.  Under the canonical identification
$u_{a+F}=u_a u_F$, the Cartan formula expresses the total operation on
$u_{a+F}$ as the product of the total operations on $u_a$ and $u_F$.
The total characteristic class of $F$ has constant term $1$ and is
therefore invertible in the graded Chow ring; multiplication by its
inverse gives the same formula with the virtual Chern classes of $a$.
At the prime $3$, the class in
bidegree $(4,2)$ is the first Newton class
$c_1(a)^2-2c_2(a)$, which is additive on $K_0(X)$.
\end{proof}

\begin{theorem}\label{powerobstructions}
Let $X$ be a smooth projective connected surface and suppose that
\[
  X_+\simeq S^{0,0}\vee F\vee S^{4,2}
\]
in $\SH(k)$.  If $\operatorname{char}(k)\ne2$, then
\[
  \overline{c_1(T_X)}=0\in\CH^1(X)/2,
  \qquad
  \overline{c_2(T_X)}=0\in\CH^2(X)/2.
\]
If $\operatorname{char}(k)\ne3$, then
\[
  \overline{c_1(T_X)^2-2c_2(T_X)}=0\in\CH^2(X)/3.
\]
\end{theorem}

\begin{proof}
Dualizing gives
\[
  \mathcal D(X_+)\simeq S^{0,0}\vee\mathcal D(F)\vee S^{-4,-2}.
\]
For $q=2$ or $3$, the Thom isomorphism identifies
\[
  \widetilde H^{-4,-2}(\mathcal D(X_+);\Z/q)
  \simeq H^{0,0}(X;\Z/q)\simeq\Z/q.
\]
The sphere $S^{-4,-2}$ already contributes this copy of $\Z/q$;
hence the group in this bidegree coming from $\mathcal D(F)$ is zero,
and the Thom generator is supported on the sphere summand.  Its
positive reduced powers vanish because
\[
  H^{2,1}(k;\Z/2)=0,
  \qquad H^{4,2}(k;\Z/q)=0.
\]
Applying Lemma~\ref{thomwu} to $a=-T_X$ and using
\[
  c_1(-T_X)=-c_1(T_X),
  \qquad c_2(-T_X)=c_1(T_X)^2-c_2(T_X)
\]
gives the two mod-$2$ assertions successively.  For $q=3$ one obtains
\[
  c_1(-T_X)^2-2c_2(-T_X)
  =2c_2(T_X)-c_1(T_X)^2,
\]
which is equivalent to the stated congruence.  This uses only the
naturality and additivity of the stable operations, not a
multiplicative wedge decomposition.
\end{proof}

\begin{remark}\label{rmk01}
Theorem~\ref{powerobstructions} obstructs a splitting of the displayed form for
$\mathbb P^2$, since $c_1(T_{\mathbb P^2})=3H$ is not divisible by
$2$.  Under the exact identities of Theorem~\ref{chernsufficient}, the remaining
congruences hold automatically.
\end{remark}

\begin{theorem}\label{surfacecriterion}
Let $k$ be a perfect field of characteristic different from $2$, and
let $X$ be a smooth projective connected surface over $k$ with a
$k$-rational point.  Suppose that
\[
  \CH^2(X)[2]=0,
  \qquad 2c_2(T_X)=c_1(T_X)^2\quad\text{in }\CH^2(X).
\]
Then the following are equivalent:
\begin{enumerate}
\item $X_+\simeq S^{0,0}\vee F\vee S^{4,2}$ in $\SH(k)$ for some
      $F$;
\item $c_1(T_X)$ is divisible by $2$ in $\CH^1(X)$.
\end{enumerate}
\end{theorem}

\begin{proof}
If $c_1(T_X)=2\ell$, choose a line bundle $L$ with $c_1(L)=\ell$.
The assumed identity gives
\[
  2\bigl(c_2(T_X)-2\ell^2\bigr)=0.
\]
Since $\CH^2(X)$ has no $2$-torsion, the exact identities of
Theorem~\ref{chernsufficient} hold, and the splitting follows.  Conversely, a
splitting forces $c_1(T_X)$ to vanish in $\CH^1(X)/2$ by
Theorem~\ref{powerobstructions}; equivalently, it is divisible by $2$.
\end{proof}

\begin{corollary}\label{hirzebruch-splitting}
Let $k$ be a perfect field of characteristic different from $2$, and
let
\[
  \mathbb F_n=\mathbb P_{\mathbb P^1_k}
  (\mathcal O\oplus\mathcal O(n)),\qquad n\ge0,
\]
be the split Hirzebruch surface.  Then
\[
  (\mathbb F_n)_+\simeq S^{0,0}\vee E_n\vee S^{4,2}
  \quad\text{in }\SH(k)
\]
for some $E_n$ if and only if $n$ is even.
\end{corollary}

\begin{proof}
Let $C_0$ be the section with $C_0^2=-n$ and let $f$ be the class of
a fiber.  Then
\[
  \Pic(\mathbb F_n)=\Z C_0\oplus\Z f,
  \qquad C_0\cdot f=[\mathrm{pt}],
  \qquad f^2=0,
\]
and its Chern classes are
\[
  c_1(T_{\mathbb F_n})=2C_0+(n+2)f,
  \qquad c_2(T_{\mathbb F_n})=4[\mathrm{pt}].
\]
It follows that $c_1(T_{\mathbb F_n})^2=8[\mathrm{pt}]
=2c_2(T_{\mathbb F_n})$.  Moreover,
$\CH^2(\mathbb F_n)\simeq\Z$ and $\mathbb F_n$ has a rational point.
Finally, $2C_0+(n+2)f$ is divisible by $2$ in the Picard group exactly
when $n$ is even.  Theorem~\ref{surfacecriterion} gives the result.
\end{proof}

\section{The \texorpdfstring{$SL^c$}{SLc}-oriented Milnor--Witt
realization}
\label{sectionK}

In this section, we use $SL^c$ Thom isomorphisms to construct the two
extreme summands in Milnor--Witt motives and then compare them with the
image of a stable splitting.  From now on, $k$ is an infinite perfect
field of characteristic different from $2$.  We write $\DMW(k)$ for
the stable category of Milnor--Witt motives, $\widetilde M(X)$ for the
motive of a smooth $k$-scheme, and $\widetilde\Z$ for the tensor unit.
Our convention is
\[
  \Sigma_{\T}\widetilde\Z=\widetilde\Z(1)[2].
\]
We refer to~\cite{bachmann2022milnorwitt} for the construction,
cancellation theorem, and computations of Milnor--Witt motivic
cohomology used below.

The strong symmetric monoidal motive functor
$\Phi:\SH(k)\to\DMW(k)$ preserves duals and carries stable
decompositions forward~\cite[Chapter 3, equation~3.3.6.a, and Chapter 6,
Section 5]{bachmann2022milnorwitt}.  In the absence of a stable Thom
trivialization, an $SL^c$-orientation still supplies the Thom
isomorphism needed for the same extreme-retract construction in
$\DMW(k)$ once a compatible stable presentation has been fixed.  We
first construct this oriented realization and then compare it with the
image under $\Phi$ of an actual stable splitting.

\subsection{Categorical $SL^c$ Thom isomorphisms}

\begin{definition}
An $SL^c$-structure on a rank-$r$ vector bundle $E$ on a smooth
scheme $Y$ consists of a line bundle $L$ and an isomorphism
\[
  \lambda:\det(E)\xrightarrow{\ \simeq\ }L^{\otimes2}.
\]
When these data are chosen, we call $E$ $SL^c$-oriented.  Their mere
existence is called $SL^c$-orientability.
\end{definition}

\begin{proposition}
\label{slcthom}
Let $(E,L,\lambda)$ be a rank-$r$ vector bundle with an
$SL^c$-structure on a smooth $k$-scheme $Y$.  There is a functorial
isomorphism
\[
  \widetilde M(\Th(E))\xrightarrow{\ \simeq\ }
  \widetilde M(Y)(r)[2r].
\]
\end{proposition}

\begin{proof}
This is the categorical $SL^c$ Thom isomorphism of
Ananyevskiy~\cite[Corollary 3.10 and Remark 3.11]{Ananyevskiy2021}.
\end{proof}

The cited theorem provides normalized Thom isomorphisms that are
functorial under pullback and isomorphisms of $SL^c$-bundles.
Multiplicativity under direct sums is the additional orientation
property distinguished in \cite[Remark 3.4]{Ananyevskiy2021}.  We
therefore keep track of the stable presentation used to turn the
virtual bundle $-T_X$ into an honest $SL^c$-bundle.

\begin{definition}\label{def:stable-presentation}
Let $X$ be smooth projective and let
\(
  \lambda:\det(T_X)\xrightarrow{\simeq}L^{\otimes2}
\)
be an $SL^c$-orientation.  A \emph{$\lambda$-compatible stable
presentation} is a tuple
\[
  \sigma=(v:U\to X,E,N,\beta),
  \qquad
  \beta:v^*T_X\oplus E\xrightarrow{\simeq}\mathcal O_U^{\oplus N},
\]
where $v$ is a Jouanolou affine bundle and $E$ is a vector bundle on
$U$.  Taking determinants in $\beta$ and using $v^*\lambda$ equips
$E$ with the $SL^c$-structure
\[
  \mu_\sigma:\det(E)\xrightarrow{\simeq}
  (v^*L^\vee)^{\otimes2}.
\]
Such presentations exist: $v$ is an $\mathbb A^1$-equivalence, and on
the affine scheme $U$ the bundle $v^*T_X$ is a direct summand of a
finite-rank trivial bundle.
\end{definition}

For $A\in\DMW(k)$, write
$D(A)=\underline{\Hom}(A,\widetilde\Z)$.

\begin{proposition}\label{mwduality}
Let $X$ be a smooth projective $k$-scheme of pure dimension $d$.  If
$T_X$ has a chosen $SL^c$-orientation
\[
  \lambda:\det(T_X)\xrightarrow{\ \simeq\ }L^{\otimes2},
\]
then every $\lambda$-compatible stable presentation $\sigma$
determines an isomorphism
\[
  \delta_{\lambda,\sigma}:D(\widetilde M(X))
  \xrightarrow{\ \simeq\ }
  \widetilde M(X)(-d)[-2d].
\]
\end{proposition}

\begin{proof}
Motivic Atiyah duality gives
\[
  \mathcal D(\Sigma^\infty_{\T}X_+)\simeq X^{-T_X}
\]
in $\SH(k)$.  The symmetric monoidal Milnor--Witt motive functor
preserves the dual of the strongly dualizable object $X_+$, and hence
\begin{equation}\label{eq:mw-atiyah}
  D(\widetilde M(X))\simeq\widetilde M(X^{-T_X}).
\end{equation}

Write $\sigma=(v:U\to X,E,N,\beta)$.  The morphism $v$ is an
$\mathbb A^1$-equivalence, and $\beta$ gives
\[
  v^*T_X\oplus E\simeq\mathcal O_U^{\oplus N}.
\]
The bundle $E$ has rank $N-d$ and
\[
  \det(E)\simeq\det(v^*T_X)^{-1}
  \simeq(v^*L^{-1})^{\otimes2},
\]
and its $SL^c$-orientation is precisely $\mu_\sigma$.  In the stable
motivic category,
\[
  U^{-v^*T_X}\simeq\Th(E)\wedge S^{-2N,-N}.
\]
Proposition~\ref{slcthom} therefore gives
\begin{align*}
  \widetilde M(X^{-T_X})
  &\simeq\widetilde M(\Th(E))(-N)[-2N]\\
  &\simeq\widetilde M(U)(N-d)[2(N-d)](-N)[-2N]\\
  &\simeq\widetilde M(X)(-d)[-2d].
\end{align*}
Together with~\eqref{eq:mw-atiyah}, this proves the assertion.
We denote this composite by $\delta_{\lambda,\sigma}$, retaining the
chosen presentation in the notation.  If $\sigma'$ is another
presentation, then the two resulting self-dualities differ by the uniquely determined
automorphism
\[
  a_{\sigma',\sigma}=
  \bigl(\delta_{\lambda,\sigma'}\circ
  \delta_{\lambda,\sigma}^{-1}\bigr)(d)[2d]
  \in\operatorname{Aut}_{\DMW(k)}(\widetilde M(X)).
\]
\end{proof}

\subsection{Construction of the extreme Milnor--Witt projectors}

\begin{theorem}\label{mwextremesplitting}
Let $X$ be a smooth projective connected $k$-scheme of pure dimension
$d>0$.  Fix a point $e\in X(k)$ and an $SL^c$-orientation
\[
  \lambda:\det(T_X)\xrightarrow{\ \simeq\ }L^{\otimes2}.
\]
Fix also a $\lambda$-compatible stable presentation $\sigma$.
Then there is a decomposition
\begin{equation}\label{eq:mw-decomposition}
  \widetilde M(X)\simeq
  \widetilde\Z\oplus
  \widetilde M_{\mathrm{mid}}(X,e,\lambda,\sigma)\oplus
  \widetilde\Z(d)[2d]
\end{equation}
in $\DMW(k)$.  The middle summand is geometric, equivalently compact
in $\DMW(k)$.
\end{theorem}
\begin{proof}
Put $M=\widetilde M(X)$, and let $p:X\to\Spec(k)$ be the structure
morphism.  The point and structure morphism induce maps
\[
  i_0=\widetilde M(e):\widetilde\Z\longrightarrow M,
  \qquad
  r_0=\widetilde M(p):M\longrightarrow\widetilde\Z,
\]
with $r_0i_0=\id_{\widetilde\Z}$.

Let
\[
  \delta_{\lambda,\sigma}:D(M)\xrightarrow{\ \simeq\ }M(-d)[-2d]
\]
be the isomorphism of Proposition~\ref{mwduality}.  Dualizing $i_0$
and $r_0$ and then twisting by $(d)[2d]$ defines
\begin{align*}
  i_{2d}
  &=\bigl(\delta_{\lambda,\sigma}\circ D(r_0)\bigr)(d)[2d]:
    \widetilde\Z(d)[2d]\longrightarrow M,\\
  r_{2d}
  &=\bigl(D(i_0)\circ\delta_{\lambda,\sigma}^{-1}\bigr)(d)[2d]:
    M\longrightarrow\widetilde\Z(d)[2d].
\end{align*}
Contravariance of duality gives
\[
  r_{2d}i_{2d}=D(i_0)D(r_0)=D(r_0i_0)=\id.
\]

For $d>0$, Milnor--Witt motivic cohomology gives
\begin{align*}
  \Hom_{\DMW(k)}(\widetilde\Z,\widetilde\Z(d)[2d])
  &\simeq\widetilde\CH^d(\Spec(k))=0,\\
  \Hom_{\DMW(k)}(\widetilde\Z(d)[2d],\widetilde\Z)
  &\simeq H^{-d}_{\mathrm{Zar}}(\Spec(k),W)=0;
\end{align*}
  see~\cite[Chapter 3, Theorem 4.2.4 and Corollary
  4.2.6]{bachmann2022milnorwitt}.  Consequently,
$r_0i_{2d}=0=r_{2d}i_0$.  Define
\[
  q_0=i_0r_0,
  \qquad q_{2d}=i_{2d}r_{2d}\in\End_{\DMW(k)}(M).
\]
These are orthogonal idempotents, so
\[
  q_{\mathrm{mid}}=\id_M-q_0-q_{2d}
\]
is also an idempotent.  Since $\DMW(k)$ is idempotent complete, the
  three images give~\eqref{eq:mw-decomposition}.  Their images are
retracts of the compact object $M$, so the middle summand is compact.
In the stable category of Milnor--Witt motives, compact objects are
precisely the geometric ones: this follows from the effective
statement~\cite[Chapter 3, Proposition
3.2.23]{bachmann2022milnorwitt} and the fully faithful stabilization
of~\cite[Chapter 4, Corollary 4.0.2]{bachmann2022milnorwitt}.
\end{proof}

\begin{corollary}\label{K3splitting1}
Let $X$ be a smooth projective connected $k$-scheme of dimension
$d>0$, with a rational point and trivial canonical bundle.  After a
choice of trivialization of $\omega_{X/k}$ and a compatible stable
presentation $\sigma$, there is a decomposition
\[
  \widetilde M(X)\simeq
  \widetilde\Z\oplus\widetilde M_{\mathrm{mid}}(X,\sigma)
  \oplus\widetilde\Z(d)[2d].
\]
For a surface, the last summand is $\widetilde\Z(2)[4]$.
\end{corollary}

\begin{proof}
The isomorphism
$\det(T_X)\simeq\omega_{X/k}^{-1}\simeq\mathcal O_X$ gives the
required $SL^c$-orientation with square-root line bundle
$\mathcal O_X$.
\end{proof}

\begin{corollary}\label{complexK3MW}
If $X$ is a smooth projective connected complex variety of dimension
$d>0$ with trivial canonical bundle, then, after choosing a
trivialization and a compatible stable presentation $\sigma$, one has
\[
  \widetilde M(X)\simeq
  \widetilde\Z\oplus\widetilde M_{\mathrm{mid}}(X,\sigma)
  \oplus\widetilde\Z(d)[2d].
\]
In particular, this applies to complex K3 surfaces and abelian
varieties.
\end{corollary}

\begin{corollary}\label{curveMW}
Let $C$ be a smooth projective connected curve over $k$ with a
rational point.  If $C$ admits a theta characteristic
$\theta^{\otimes2}\simeq\omega_{C/k}$, then, after choosing a
compatible stable presentation $\sigma$, one has
\[
  \widetilde M(C)\simeq
  \widetilde\Z\oplus\widetilde M_{\mathrm{mid}}(C,\sigma)
  \oplus\widetilde\Z(1)[2].
\]
\end{corollary}

\begin{proof}
The dual line bundle $\theta^\vee$ is a square root of
$T_C\simeq\omega_{C/k}^{-1}$.
\end{proof}

\subsection{Compatibility with stable extreme projectors}

Let $\Phi:\SH(k)\to\DMW(k)$ denote the strong symmetric monoidal
Milnor--Witt motive functor.  The following proposition makes the
forward comparison of the two constructions precise.  It also records
the dependence on the chosen Thom trivialization.

\begin{proposition}
\label{stableMWcompat}
Let $X$ be as in Theorem~\ref{mwextremesplitting}, put
$A=\Sigma^\infty_{\T}X_+$, and let
$i:\one\to A$ and $r:A\to\one$ be induced by the point and the
structure morphism.  Write $q_0^{\MW}$ and
$q_{2d}^{\lambda,\sigma}$ for the
projectors constructed in Theorem~\ref{mwextremesplitting}.  Suppose
that a Thom equivalence
\[
  \vartheta:\Th_X(T_X-\mathcal O_X^{\oplus d})
  \xrightarrow{\ \simeq\ }\one_X
\]
has been chosen.  Let
\[
  \delta_\vartheta:\mathcal D(A)\xrightarrow{\ \simeq\ }
  A\wedge S^{-2d,-d}
\]
be the resulting self-duality, and put $T=S^{2d,d}$.  Define
\begin{align*}
  j_\vartheta
  &=\bigl(\delta_\vartheta\circ\mathcal D(r)\bigr)\wedge T:
    T\longrightarrow A,\\
  s_\vartheta
  &=\bigl(\mathcal D(i)\circ\delta_\vartheta^{-1}\bigr)\wedge T:
    A\longrightarrow T,
\end{align*}
and
\[
  \overline j_\vartheta
  =j_\vartheta-i(rj_\vartheta).
\]
Then
\[
  q_0^{\SH}=ir,
  \qquad
  q_{2d}^{\SH}=\overline j_\vartheta s_\vartheta
\]
are orthogonal idempotents.  Moreover,
\[
  \Phi(q_0^{\SH})=q_0^{\MW},
  \qquad
  \Phi(q_{2d}^{\SH})=q_{2d}^{\Phi(\delta_\vartheta)},
\]
where the right-hand side is the top projector obtained by applying
the construction of Theorem~\ref{mwextremesplitting} to the self-duality
$\Phi(\delta_\vartheta)$.

Under $\DMW(k)\to\DM(k)$, these two idempotents are the standard
covariant Chow projectors $[X\times e]$ and $[e\times X]$.  If
$\delta_{\lambda,\sigma}$ is the self-duality arising from the chosen
$SL^c$-orientation and stable presentation, then for
\[
  a=\bigl(\Phi(\delta_\vartheta)\circ
  \delta_{\lambda,\sigma}^{-1}\bigr)(d)[2d]
  \in\operatorname{Aut}_{\DMW(k)}(\widetilde M(X))
\]
one has
\[
  q_{2d}^{\Phi(\delta_\vartheta)}
  =a q_{2d}^{\lambda,\sigma}a^{-1}.
\]
We call $\vartheta$ \emph{Milnor--Witt compatible} with
$(\lambda,\sigma)$ if
there is a unit $u\in\operatorname{GW}(k)^\times$ such that, for
$M=\widetilde M(X)$,
\[
  \Phi(\delta_\vartheta)
  =\bigl((u\,\id_M)(-d)[-2d]\bigr)
   \circ\delta_{\lambda,\sigma}.
\]
In that case the two top projectors agree.
\end{proposition}

\begin{proof}
Contravariance of duality gives $s_\vartheta j_\vartheta=\id_T$.
Stable connectivity gives $s_\vartheta i=0$, while $r i=\id_\one$.
Consequently,
\[
  r\overline j_\vartheta=0,
  \qquad s_\vartheta\overline j_\vartheta=\id_T,
\]
which proves the first assertion.

The functor $\Phi$ preserves duals, evaluation and coevaluation maps,
and all compositions used above.  Furthermore,
\[
  \Hom_{\DMW(k)}(\widetilde\Z(d)[2d],\widetilde\Z)=0,
\]
see~\cite[Chapter 3, Theorem 4.2.4 and Corollary 4.2.6]
{bachmann2022milnorwitt}.  Thus $\Phi(rj_\vartheta)=0$ and
$\Phi(\overline j_\vartheta)=\Phi(j_\vartheta)$.  Under the canonical
monoidal identifications
\[
  \Phi(A)=\widetilde M(X),\qquad
  \Phi(T)=\widetilde\Z(d)[2d],\qquad
  \Phi(\mathcal D(A))=D(\widetilde M(X)),
\]
the displayed formulas are exactly those of
Theorem~\ref{mwextremesplitting} for the induced self-duality.

Let $\Psi:\DMW(k)\to\DM(k)$ be the comparison motive functor.  Make the
comparison first in the relative category $\DM(X)$.  The image of
$\vartheta$ and the canonical oriented Thom isomorphism are
isomorphisms between the same invertible objects of $\DM(X)$; their
ratio is therefore an automorphism of the tensor unit.  Thus it is a
unit in
\[
  \End_{\DM(X)}(\one_X)=H^{0,0}(X,\Z).
\]
Since $X$ is connected, this unit is a sign
$\varepsilon\in\{\pm1\}$.  After applying the structural pushforward,
it multiplies
the top inclusion and retraction by $\varepsilon$ and
$\varepsilon^{-1}$, respectively, and hence does not change their
composite.  Thus $\Psi\Phi(q_0^{\SH})$ and
$\Psi\Phi(q_{2d}^{\SH})$ are the standard covariant Chow projectors
$[X\times e]$ and $[e\times X]$.

Finally,
\[
  \Phi(\delta_\vartheta)
  =\bigl(a(-d)[-2d]\bigr)\circ\delta_{\lambda,\sigma}.
\]
Substitution in the definitions gives
$j_{\Phi(\delta_\vartheta)}=a j_{\lambda,\sigma}$ and
$s_{\Phi(\delta_\vartheta)}=s_{\lambda,\sigma}a^{-1}$, proving the conjugacy
formula.  Under the compatibility hypothesis, $a=u\,\id_M$ is
central, so the two top projectors agree.
\end{proof}

\begin{corollary}\label{K3MWforwardimage}
Let $X$ be a complex projective K3 surface.  Fix a Beauville--Voisin
point $o$, a Thom equivalence $\vartheta$ as in
Corollary~\ref{K3SHsplitting}, and the extreme-cell splitting constructed
from $(o,\vartheta)$:
\[
  X_+\simeq\one\vee F_{X,o,\vartheta}\vee S^{4,2}.
\]
Its Milnor--Witt realization is a decomposition
\[
  \widetilde M(X)\simeq
  \widetilde\Z\oplus\Phi(F_{X,o,\vartheta})
  \oplus\widetilde\Z(2)[4].
\]
The extreme idempotents have ordinary Chow images $[X\times o]$ and
$[o\times X]$.  Hence the ordinary motivic image of
$\Phi(F_{X,o,\vartheta})$ is the middle Chow-motive summand cut out by
\[
  [\Delta_X]-[X\times o]-[o\times X].
\]
For every separately chosen $SL^c$-orientation
$(\lambda,\sigma)$, the realized top projector is conjugate to
$q_4^{\lambda,\sigma}$; they agree whenever $\vartheta$ is
Milnor--Witt compatible with $(\lambda,\sigma)$.
\end{corollary}

\section{Chow--Witt cycles and projector relations}\label{sec:cycles}

In this section, we treat the cycle-level lifting problem directly.
For the standard extreme projectors, normalized Gysin correspondences
give a three-term Milnor--Witt decomposition without choosing the stable
presentation of Definition~\ref{def:stable-presentation}.  Comparing
this cycle construction with the presentation-dependent self-duality
of Section~\ref{sectionK} requires an additional compatibility between
the two chosen systems of Thom classes.  The nilpotence theorem needed for a
general additive system of lifted Chow projectors is recorded in
Appendix~\ref{appendix:nilpotence}.

\subsection{The cycle obstruction}

We begin by recalling the obstruction to lifting an ordinary Chow
cycle to a Chow--Witt cycle.

Let $Y$ be a smooth $k$-scheme, let $\mathcal L$ be a line bundle on
$Y$, and let $n\ge0$.  The Cartesian description of Milnor--Witt
$K$-theory in terms of Milnor $K$-theory and powers of the fundamental
ideal gives an exact sequence
\begin{equation}\label{diagramofchowwitt}
\begin{split}
  H^n_{\mathrm{Nis}}(Y,I^{n+1}(\mathcal L))
  &\longrightarrow\widetilde\CH^n(Y,\mathcal L)
  \longrightarrow\CH^n(Y)\\
  &\xrightarrow{\ \partial_{\mathcal L}\ }
  H^{n+1}_{\mathrm{Nis}}(Y,I^{n+1}(\mathcal L)).
\end{split}
\end{equation}
See~\cite{MSMF_2008_2_113__1_0,bachmann2022milnorwitt}.  Thus a
cycle $\alpha\in\CH^n(Y)$ has a Chow--Witt lift with twist
$\mathcal L$ if and only if
$\partial_{\mathcal L}(\alpha)=0$.

Put $\Ch^*(Y)=\CH^*(Y)/2$, and let
\[
  \rho_{\mathcal L}:
  H^{n+1}_{\mathrm{Nis}}(Y,I^{n+1}(\mathcal L))
  \longrightarrow\Ch^{n+1}(Y)
\]
be reduction.  The twisted Bockstein computation gives
\begin{equation}\label{eq:twisted-square}
  \rho_{\mathcal L}\partial_{\mathcal L}(\alpha)
  =Sq^2_{\mathcal L}(\overline\alpha),
  \qquad
  Sq^2_{\mathcal L}
  =Sq^2+\overline{c_1(\mathcal L)}\cup-;
\end{equation}
see~\cite[Theorem 4.1.4]{Secondary}.  Thus vanishing of the twisted
square is necessary, and it is sufficient when $\rho_{\mathcal L}$ is
injective.

For a smooth proper $d$-fold $X$, the identity in the finite
Milnor--Witt correspondence category is the graph of $\id_X$, hence
the Chow--Witt diagonal in
$\widetilde\CH^d(X\times X,p_2^*\omega_{X/k})$; see
\cite[Chapter 2, Sections 4.2--4.3, especially Definition
4.3.2]{bachmann2022milnorwitt}.  The motive functors give a forward
comparison along
\[
  \SH(k)\longrightarrow\DMW(k)\longrightarrow\DM(k).
\]
The induced maps on endomorphism rings are not generally surjective,
so a Chow--Witt projector decomposition does not automatically lift
to $\SH(k)$.  Similarly, an additive lift of an ordinary Chow
projector need not be idempotent: its square can differ from itself by
a nonzero class in the kernel of the comparison homomorphism.

\subsection{The correspondence ring and the extreme cycles}

We next identify the correspondence ring and construct the extreme
cycles.

Let $X$ be smooth projective of pure dimension $d$, and put
$M=\widetilde M(X)$.  Since $M$ is strongly dualizable, evaluation
induces an isomorphism
\[
  \End_{\DMW(k)}(M)
  \simeq
  \Hom_{\DMW(k)}(M\otimes D(M),\widetilde\Z).
\]
Milnor--Witt Poincar\'e duality and the identification of diagonal
Milnor--Witt motivic cohomology with Chow--Witt groups identify the
right-hand side naturally with
\begin{equation}\label{eq:mw-correspondence-ring}
  R_{\MW}(X)=
  \widetilde\CH^d(X\times X,p_2^*\omega_{X/k}),
\end{equation}
see~\cite[Chapter 6, Theorem 2.3.6; Chapter 3, Corollary
4.2.6; and Chapter 4, Corollary
4.0.2]{bachmann2022milnorwitt}.  We record the compatibility with
composition.  For $\alpha,\beta\in R_{\MW}(X)$, let
$p_{12},p_{23},p_{13}:X^3\to X^2$ and $q_i:X^3\to X$ be the
projections.  Under the preceding duality isomorphism,
\begin{equation}\label{eq:mw-correspondence-composition}
  \beta\circ\alpha
  =(p_{13})_*
  \bigl(p_{12}^*\alpha\cdot p_{23}^*\beta\bigr).
\end{equation}
Indeed, the product has twist
$q_2^*\omega_{X/k}\otimes q_3^*\omega_{X/k}$.  Since
$q_2^*\omega_{X/k}=\omega_{p_{13}}$ and
$q_3^*\omega_{X/k}=p_{13}^*p_2^*\omega_{X/k}$, the displayed
pushforward has the twist in~\eqref{eq:mw-correspondence-ring}.
Formula~\eqref{eq:mw-correspondence-composition} follows from the
cap-product projection formula; see~\cite[Chapter 6, Remark 2.3.2 and
Chapter 2, Section 4.2]{bachmann2022milnorwitt}.  Thus the full twisted
Chow--Witt group in~\eqref{eq:mw-correspondence-ring}, rather than only
the subgroup of finite-support correspondences, is the endomorphism
ring of $M$.

Forgetting quadratic data commutes with pullback, products, and proper
pushforward.  It therefore gives a ring homomorphism
\[
  u:R_{\MW}(X)\longrightarrow\CH^d(X\times X).
\]
This is an endomorphism ring in $\DMW(k)$, not in $\SH(k)$.

For comparison, we record the associated mod-$2$ obstruction.  If
$i:Z\hookrightarrow Y$ is a smooth closed immersion, the motivic Wu
formula gives
\[
  Sq^2(i_*1)=i_*c_1(N_{Z/Y}).
\]
For a rational point $e\in X(k)$ this yields
\begin{align*}
  Sq^2[\Delta_X]&=\Delta_{X*}c_1(T_X),\\
  Sq^2[X\times e]&=0,\\
  Sq^2[e\times X]&=0.
\end{align*}
The normal bundles of the last two immersions are constant bundles
with fiber $T_eX$, hence are trivial as bundles on $X$; no orientation
hypothesis is needed for these two equalities.  After including the
twist $p_2^*\omega_{X/k}$ in~\eqref{eq:mw-correspondence-ring}, one gets
\begin{align*}
  Sq^2_{p_2^*\omega_{X/k}}[\Delta_X]&=0,\\
  Sq^2_{p_2^*\omega_{X/k}}[X\times e]&=0,\\
  Sq^2_{p_2^*\omega_{X/k}}[e\times X]
  &=(e\times\id_X)_*\overline{c_1(\omega_{X/k})}.
\end{align*}
An $SL^c$-orientation makes the last class vanish.
By~\eqref{eq:twisted-square}, these vanishings are necessary for
Chow--Witt lifts, but they prove existence only when the corresponding
reduction map is injective.  Proper Gysin maps give the three lifts
directly, without that extra assumption.

\begin{proposition}\label{extremecyclelifts}
Let $X$ be smooth projective of pure dimension $d$, let $e\in X(k)$,
and fix an $SL^c$-orientation
\[
  \lambda:\det(T_X)\xrightarrow{\ \simeq\ }L^{\otimes2}.
\]
In the covariant correspondence convention, each of
\[
  [\Delta_X],\qquad [X\times e],\qquad [e\times X]
\]
belongs to $\im(u)$.  Consequently, so does
$[\Delta_X]-[X\times e]-[e\times X]$.
\end{proposition}

\begin{proof}
Put $S=X\times X$ and $\mathcal L=p_2^*\omega_{X/k}$.  If
$i:Z\hookrightarrow S$ is a smooth closed immersion of codimension
$d$, proper pushforward in Chow--Witt theory has the form
\[
  i_*:\widetilde\CH^0
  \bigl(Z,\det(N_i)\otimes i^*\mathcal L\bigr)
  \longrightarrow\widetilde\CH^d(S,\mathcal L).
\]
For the diagonal, the source twist is canonically
\[
  \det(T_X)\otimes\omega_{X/k}\simeq\mathcal O_X.
\]
For $j_0(x)=(x,e)$, it is the constant line
\[
  \det(T_eX)\otimes\omega_{X/k,e}\simeq k.
\]
The corresponding Gysin classes therefore lift $[\Delta_X]$ and
$[X\times e]$.

For $j_{2d}(x)=(e,x)$, the source twist is
\[
  \det(T_eX)\otimes\omega_{X/k}.
\]
Write $L_e=e^*L$, viewed as a constant line bundle on $X$.  The chosen
orientation and its fiber at $e$ give
\[
  (L^\vee\otimes L_e)^{\otimes2}
  \xrightarrow{\ \lambda^\vee\otimes(e^*\lambda)^{-1}\ }
  \omega_{X/k}\otimes\det(T_eX).
\]
This square presentation orients the source twist, so pushing forward
its unit gives a Chow--Witt class lifting $[e\times X]$.  Proper
pushforward commutes with $u$, and taking the difference gives the
last assertion.
\end{proof}

An additive lift is not yet a projector.  In this special three-term
situation, however, the Gysin maps can be normalized so that all
projector relations hold before forgetting quadratic information.

\begin{proposition}
\label{cyclemwdecomposition}
Let $X$ be a smooth projective connected $k$-scheme of pure dimension
$d>0$, let $e\in X(k)$, and fix an $SL^c$-orientation $\lambda$ of
$T_X$.  The two standard extreme Chow projectors have Chow--Witt lifts
$q_0^{\mathrm{cyc}},q_{2d}^{\mathrm{cyc}}\in R_{\MW}(X)$ that are
orthogonal idempotents, with
\[
  \im(q_0^{\mathrm{cyc}})\simeq\widetilde\Z,
  \qquad
  \im(q_{2d}^{\mathrm{cyc}})\simeq\widetilde\Z(d)[2d].
\]
Consequently,
\[
  \widetilde M(X)\simeq
  \widetilde\Z\oplus
  \widetilde M^{\mathrm{cyc}}_{\mathrm{mid}}(X,e,\lambda)
  \oplus\widetilde\Z(d)[2d].
\]
\end{proposition}

\begin{proof}
Put $M=\widetilde M(X)$ and $T^d=\widetilde\Z(d)[2d]$.  In the stable
Chow--Witt correspondence description,
\[
  \Hom(T^d,M)\simeq\widetilde\CH^0(X,\omega_{X/k}),
  \qquad
  \Hom(M,T^d)\simeq\widetilde\CH^d(X).
\]
The first identification follows from Milnor--Witt Poincar\'e duality,
and the second from the identification of diagonal Milnor--Witt
motivic cohomology with Chow--Witt groups;
see~\cite[Chapter 2, Section 4.2; Chapter 3, Corollary 4.2.6; and
Chapter 4, Corollary 4.0.2]{bachmann2022milnorwitt}.  Compatibility of
these identifications with composition is the composition law for
stable Chow--Witt correspondences.  Explicitly, if
\[
  a\in\widetilde\CH^0(X,\omega_{X/k}),
  \qquad b\in\widetilde\CH^d(X),
\]
then the corresponding composite endomorphism of $T^d$ is
\begin{equation}\label{eq:chowwitt-pairing}
  p_*(a\cdot b)\in\widetilde\CH^0(k)=GW(k),
\end{equation}
where $p:X\to\Spec(k)$ is the structure morphism.  This is the
projection formula in the correspondence description; see
\cite[Chapter 2, Corollary 3.1.4 and Section 4.2]
{bachmann2022milnorwitt}.

The dual square presentation
\[
  \lambda^\vee:(L^\vee)^{\otimes2}
  \xrightarrow{\ \simeq\ }\omega_{X/k}
\]
transports the unit to a class
\[
  \mathbf 1_{\lambda^\vee}
  \in\widetilde\CH^0(X,\omega_{X/k}).
\]
Likewise, the fiber
$e^*\lambda:\det(T_eX)\simeq L_e^{\otimes2}$ determines
\[
  \mathbf 1_{e^*\lambda}
  \in\widetilde\CH^0(k,\det(T_eX)).
\]
Proper pushforward for the regular immersion $e$ gives
\[
  e_*\mathbf 1_{e^*\lambda}\in\widetilde\CH^d(X).
\]
Under the two Hom identifications above, let
\[
  i_{2d}^{\mathrm{raw}}:T^d\longrightarrow M,
  \qquad
  r_{2d}^{\mathrm{raw}}:M\longrightarrow T^d
\]
be the morphisms represented by the classes
$\mathbf 1_{\lambda^\vee}$ and
$e_*\mathbf 1_{e^*\lambda}$, respectively.  Formula
\eqref{eq:chowwitt-pairing} and the projection formula now give
\begin{equation}\label{eq:raw-chowwitt-retraction}
\begin{aligned}
  u_\lambda
  :=r_{2d}^{\mathrm{raw}}i_{2d}^{\mathrm{raw}}
  &=p_*\bigl(
      \mathbf 1_{\lambda^\vee}\cdot
      e_*\mathbf 1_{e^*\lambda}
    \bigr)\\
  &=(p\circ e)_*\bigl(
      e^*\mathbf 1_{\lambda^\vee}\cdot
      \mathbf 1_{e^*\lambda}
    \bigr)\in GW(k).
\end{aligned}
\end{equation}
The product in the last line has twist
\[
  e^*\omega_{X/k}\otimes\det(T_eX)\simeq k,
\]
and the tensor product of the square presentations
$e^*\lambda^\vee$ and $e^*\lambda$ is a square presentation of this
trivial line.  After applying the symmetry isomorphism, and relative
to the determinant-transitivity isomorphism
$\omega_e\otimes e^*\omega_p\simeq\omega_{p\circ e}=k$, where
$\omega_e=\det(T_eX)$ and $\omega_p=\omega_{X/k}$, it represents a
one-dimensional form $\langle a_\lambda\rangle$ for some
$a_\lambda\in k^\times$.  Proper pushforward along
$p\circ e=\id_{\Spec(k)}$ therefore shows that
\[
  u_\lambda=\langle a_\lambda\rangle\in GW(k)^\times.
\]
This uses the orientation-change, base-change, and projection formulas
with their determinant twists; see~\cite[Chapter~2, Sections~2.2, 3.1,
and~4.2]{bachmann2022milnorwitt}.  Only the displayed invertibility is
needed, so the construction is insensitive to the ordering convention
for the two determinant lines.

Normalize the two morphisms by setting
\[
  i_{2d}^{\mathrm{cyc}}=i_{2d}^{\mathrm{raw}},
  \qquad
  r_{2d}^{\mathrm{cyc}}=u_\lambda^{-1}r_{2d}^{\mathrm{raw}}.
\]
Equation~\eqref{eq:raw-chowwitt-retraction} then gives
$r_{2d}^{\mathrm{cyc}}i_{2d}^{\mathrm{cyc}}=\id_{T^d}$.  Consequently
\[
  q_{2d}^{\mathrm{cyc}}
  =i_{2d}^{\mathrm{cyc}}r_{2d}^{\mathrm{cyc}}
\]
is an idempotent.  Under the comparison functor to ordinary
correspondences, $\mathbf 1_{\lambda^\vee}$ maps to the ordinary unit
and $e_*\mathbf 1_{e^*\lambda}$ maps to the class of $e$.  Moreover,
$u_\lambda^{-1}$ maps to its rank, which is $1$.  Hence
$q_{2d}^{\mathrm{cyc}}$ maps to $[e\times X]$ in the covariant
convention.

Similarly, the point and structure morphism give
\[
  i_0:\widetilde\Z\longrightarrow M,
  \qquad r_0:M\longrightarrow\widetilde\Z,
  \qquad r_0i_0=\id,
\]
and $q_0^{\mathrm{cyc}}=i_0r_0$ lifts $[X\times e]$.  The two cross
Hom groups vanish by Chapter 3, Theorem 4.2.4 and Corollary 4.2.6 of
\cite{bachmann2022milnorwitt}; hence
\[
  r_0i_{2d}^{\mathrm{cyc}}=0
  =r_{2d}^{\mathrm{cyc}}i_0.
\]
Thus $q_0^{\mathrm{cyc}}$ and $q_{2d}^{\mathrm{cyc}}$ are orthogonal
idempotents.  Their complement
\[
  q_{\mathrm{mid}}^{\mathrm{cyc}}
  =\id_M-q_0^{\mathrm{cyc}}-q_{2d}^{\mathrm{cyc}}
\]
is an idempotent lifting
$[\Delta_X]-[X\times e]-[e\times X]$, and splitting these three
idempotents proves the assertion.
\end{proof}

\begin{remark}\label{rem:thom-gysin-comparison}
Proposition~\ref{cyclemwdecomposition} is independent of the stable
presentation used in Section~\ref{sectionK}.  The equality
$q_{2d}^{\mathrm{cyc}}=q_{2d}^{\lambda,\sigma}$ follows after fixing a
multiplicative $SL^c$-orientation whose Thom isomorphisms agree with the
normalized Gysin fundamental classes.  This compatibility is stronger
than the normalized Thom-isomorphism theorem cited in
Proposition~\ref{slcthom}; the decomposition and the nilpotence argument
below are independent of this additional identification.
\end{remark}

\subsection{General projector systems}

We finally pass from the two extreme projectors to a finite system of
projectors.  For a general Chow projector $\pi$ and an additive lift
$\widetilde\pi\in R_{\MW}(X)$, one knows only that
$\widetilde\pi^2-\widetilde\pi\in\ker(u)$.  Thus the explicit
factorizations used for the two extreme projectors cannot be replaced
by class-by-class lifting.  Appendix~\ref{appendix:nilpotence} proves
that finite $2$-cohomological dimension makes this kernel a
nilpotent ideal and consequently permits simultaneous correction of a
finite complete system of projectors.  The extreme three-projector
system above requires no such field-theoretic hypothesis.

\appendix

\section{Nilpotent kernels and simultaneous projector lifting}
\label{appendix:nilpotence}

\begin{lemma}\label{nilpotentlifting}
Let $f:R\to S$ be a unital ring homomorphism with nilpotent kernel.
Every finite complete system of orthogonal idempotents in $\im(f)$
lifts to a complete system of orthogonal idempotents in $R$.
\end{lemma}

\begin{proof}
Replace $S$ by $\im(f)$ and write $I=\ker(f)$.  If $I^2=0$ and $a$
lifts an idempotent, then $\delta=a^2-a\in I$ commutes with $a$, and
\[
  e=a+(1-2a)\delta
\]
satisfies $e^2=e$ because $\delta^2=0$.  Applying this correction to
$R/I^{r+1}\to R/I^r$ lifts a single idempotent for arbitrary
nilpotent $I$.

After lifting the first idempotent to $e_1$, the surjective corner map
\[
  (1-e_1)R(1-e_1)\longrightarrow
  (1-f(e_1))S(1-f(e_1))
\]
has nilpotent kernel.  Induction inside this corner lifts all but the
last member; the complement of their sum lifts the last and is
orthogonal to the others on both sides.
\end{proof}

\begin{theorem}\label{mwkernelnilpotent}
Let $k$ be an infinite perfect field of characteristic different from
$2$ with $\operatorname{cd}_2(k)=c<\infty$, and let $X$ be smooth
projective over $k$ of pure dimension $d$.  For
\[
  u:R_{\MW}(X)\longrightarrow\CH^d(X\times X),
\]
the ideal $\mathcal J=\ker(u)$ satisfies
\[
  \mathcal J^{c+d+1}=0.
\]
The power is taken with respect to correspondence composition.
\end{theorem}

\begin{proof}
Put $S=X\times X$ and $\mathcal L=p_2^*\omega_{X/k}$.  For $q\ge1$,
set
\[
  F^qR_{\MW}(X)=
  \im\!\left(
    H^d_{\mathrm{Nis}}(S,I^{d+q}(\mathcal L))
    \longrightarrow R_{\MW}(X)
  \right).
\]
The exact sequence~\eqref{diagramofchowwitt} gives
$\mathcal J=F^1R_{\MW}(X)$.  If
\[
  \alpha\in H^d_{\mathrm{Nis}}(S,I^{d+p}(\mathcal L)),
  \qquad
  \beta\in H^d_{\mathrm{Nis}}(S,I^{d+q}(\mathcal L)),
\]
then the product and proper-pushforward formulas in $I$-cohomology
\cite[Corollary 4.6 and Theorem 4.16]{Fas07},
\cite[Corollary 10.4.5]{MSMF_2008_2_113__1_0}, and
\cite[Chapter 2, Sections 3 and 4.2]{bachmann2022milnorwitt} give
\[
  p_{12}^*\alpha\cdot p_{23}^*\beta\in
  H^{2d}_{\mathrm{Nis}}\!\left(
    X^3,I^{2d+p+q}
    (q_2^*\omega_{X/k}\otimes q_3^*\omega_{X/k})
  \right).
\]
Here
$q_2^*\omega_{X/k}=\omega_{p_{13}}$ and
$q_3^*\omega_{X/k}=p_{13}^*\mathcal L$; hence pushforward along
$p_{13}$ gives
\[
  (p_{13})_*(p_{12}^*\alpha\cdot p_{23}^*\beta)
  \in H^d_{\mathrm{Nis}}
  (S,I^{d+p+q}(\mathcal L)).
\]
Thus
\[
  F^pR_{\MW}(X)\circ F^qR_{\MW}(X)
  \subset F^{p+q}R_{\MW}(X).
\]
Therefore $\mathcal J^n\subset F^nR_{\MW}(X)$.  Since $\dim S=2d$,
\cite[Proposition 5.1]{threefolds} gives
$I^m(\mathcal L)=0$ for $m\ge c+2d+1$.  Taking $n=c+d+1$ yields
$F^nR_{\MW}(X)=0$.
\end{proof}

\begin{corollary}\label{simultaneousprojectors}
Under the hypotheses of Theorem~\ref{mwkernelnilpotent}, every finite
complete system of orthogonal Chow projectors contained in $\im(u)$
admits a simultaneous lift to orthogonal idempotents in
$R_{\MW}(X)$.  It therefore induces a direct-sum decomposition of
$\widetilde M(X)$ into geometric Milnor--Witt motives.  If $d=2$ and
$\operatorname{cd}_2(k)=0$, then $\mathcal J^3=0$.
\end{corollary}

\begin{proof}
Apply Lemma~\ref{nilpotentlifting} to
$R_{\MW}(X)\to R_{\MW}(X)/\mathcal J\simeq\im(u)$.  The geometric
Milnor--Witt category is idempotent complete, and the resulting
summands are retracts of the compact object $\widetilde M(X)$.
\end{proof}

\section{A supplementary integral Chow-motive result}
\label{appendix:ppav}

We conclude with a supplementary integral Chow-motive decomposition
for principally polarized abelian surfaces.  We first recall how this
result enters the Milnor--Witt lifting problem.  The rank map from
Milnor--Witt to finite correspondences is compatible with exterior
products and composition.  After $\mathbb A^1$-localization and Tate
stabilization it induces a strong symmetric monoidal functor
\[
  \Psi=L\pi_*:\DMW(k)\longrightarrow\DM(k).
\]
It sends $\widetilde M(X)$ to $M(X)$ and
$\widetilde\Z(n)[2n]$ to $\Z(n)[2n]$
\cite[Chapter~3, Corollary~1.2.15 and Diagram~(3.2.4.a)]
{bachmann2022milnorwitt}, and on endomorphism rings induces
\[
  u:R_{\MW}(X)\longrightarrow\CH^d(X\times X)
\]
of~\eqref{eq:mw-correspondence-ring}.  Hence a Milnor--Witt lift of
the projectors below yields their integral Chow-motive decomposition;
the converse requires their classes to lie in $\im(u)$ because
$\Psi$ is not generally full.  We use contravariant Chow
correspondences below, obtained from the convention of
Section~\ref{sec:cycles} by transposition; thus the endpoint projectors
of $(X,e)$ are $[e\times X]$ and $[X\times e]$.

Murre constructs rational Picard and Albanese projectors and completes
them, for a surface, to five orthogonal rational projectors with the
prescribed K\"unneth actions
\cite[Section~3, Theorem~1; Section~7, Theorem~3]{Murre1990}.  We use
the calculation only when its isogeny is an isomorphism, so no
denominator occurs.  Kahn's
\'etale-integral Deninger--Murre projectors
\cite[Theorem~1.7]{Kahn2026DividedPowers} have a different scope.  The
result here concerns ordinary integral Chow groups of abelian surfaces;
$h(X)$ denotes the corresponding integral Chow motive.

\begin{lemma}\label{lem:integral-murre-curve}
Let $X$ be a smooth projective connected surface over an algebraically
closed field, let $e\in X(k)$, and let
$i:(C,c)\hookrightarrow(X,e)$ be a pointed closed immersion of a smooth
projective connected curve.  Using the canonical principal
polarization to identify $\Pic^0(C)$ and $\operatorname{Alb}(C)$ with
$J(C)$, put
\[
  \Psi_i=\operatorname{Alb}(i)\circ\Pic^0(i):
  \Pic^0(X)\longrightarrow\operatorname{Alb}(X).
\]
If $\Psi_i$ is an isomorphism, then there are mutually orthogonal
idempotents $\pi_j\in\CH^2(X\times X)$, $0\le j\le4$, such that
\[
  [\Delta_X]=\sum_{j=0}^4\pi_j,\qquad
  \pi_0=[e\times X],\qquad \pi_4=[X\times e].
\]
For every Weil cohomology theory, $\pi_j$ is the $j$-th K\"unneth
projector.
\end{lemma}

\begin{proof}
Let $a_X:X\to\operatorname{Alb}(X)$ be based at $e$, put
$\beta=\Psi_i^{-1}$, and let $\mathcal P_X$ be the Poincar\'e line
bundle on $\Pic^0(X)\times X$, normalized along
$\{0\}\times X$ and $\Pic^0(X)\times\{e\}$.  Define
\[
  D_\beta
  =c_1\bigl((\beta\circ a_X\circ p_1,p_2)^*\mathcal P_X\bigr)
  \in\CH^1(X\times X),\qquad
  F_\beta=D_\beta\circ\Gamma_i\in\CH^1(C\times X),
\]
and
\[
  E=F_\beta\circ{}^t\!\Gamma_i
   =D_\beta\circ\Gamma_i\circ{}^t\!\Gamma_i,\qquad
  Q={}^t\!E.
\]
The normalization of $\mathcal P_X$ and $i(c)=e$ give
\begin{equation}\label{eq:murre-F-normalization}
  F_\beta|_{\{c\}\times X}=0,\qquad
  F_\beta|_{C\times\{e\}}=0,\qquad
  (p_X)_*F_\beta=0.
\end{equation}
The last equality holds because the restriction to each fibre of
$p_X:C\times X\to X$ is algebraically trivial and hence has degree
zero.

The see-saw theorem and the universal property of $J(C)$ give the
integral isomorphism
\begin{equation}\label{eq:normalized-divisors}
  \ker\!\left(\Pic(C\times X)\to
    \Pic(\{c\}\times X)\oplus\Pic(C\times\{e\})\right)
  \xrightarrow{\ \simeq\ }
  \Hom(J(C),\Pic^0(X)).
\end{equation}
A normalized line bundle gives the homomorphism on the right; pulling
back the normalized Poincar\'e bundle gives the inverse.  Thus
\eqref{eq:normalized-divisors} is injective without tensoring with
$\mathbb Q$.

Under~\eqref{eq:normalized-divisors}, $F_\beta$ corresponds to
$h_\beta=\beta\circ\operatorname{Alb}(i)$.  Define the divisor
correspondence
\[
  G_\beta
  =F_\beta\circ{}^t\!\Gamma_i\circ F_\beta
  \in\CH^1(C\times X).
\]
Refined base change restricts the rightmost, respectively leftmost,
factor of this composition along $\{c\}\times X$, respectively
$C\times\{e\}$.  Hence~\eqref{eq:murre-F-normalization} gives
\[
  G_\beta|_{\{c\}\times X}=0,
  \qquad
  G_\beta|_{C\times\{e\}}=0.
\]
Thus $G_\beta$ is normalized, and its homomorphism under
\eqref{eq:normalized-divisors} is
\[
  h_\beta\circ\Pic^0(i)\circ h_\beta
  =\beta\circ\Psi_i\circ\beta\circ\operatorname{Alb}(i)
  =h_\beta.
\]
Integral injectivity gives $G_\beta=F_\beta$, so associativity gives
\[
  E^2
  =G_\beta\circ{}^t\!\Gamma_i
  =F_\beta\circ{}^t\!\Gamma_i
  =E,
\]
as in~\cite[Section~3, Theorem~1]{Murre1990}; hence $Q^2=Q$.

For the endpoint relations write $E=(i\times\id_X)_*F_\beta$,
$j_0=e\times\id_X$, and $j_4=\id_X\times e$.  Excess intersection for
$j_0$ and transverse base change for $j_4$ give
\[
  j_0^*E=c_1(\mathcal E)F_\beta|_{\{c\}\times X}=0,
  \qquad
  j_4^*E=i_*F_\beta|_{C\times\{e\}}=0,
\]
where $\mathcal E$ is the excess normal line.  If $p_1,p_2$ are the
projections of $X\times X$, the composition formula gives
\begin{align*}
  \pi_0E&=(j_4^*E)\times[X],&
  E\pi_4&=[X]\times(j_0^*E),\\
  E\pi_0&=[e]\times(p_2)_*E,&
  \pi_4E&=(p_1)_*E\times[e].
\end{align*}
These vanish because $(p_2)_*E=(p_X)_*F_\beta=0$ by
\eqref{eq:murre-F-normalization} and $(p_1)_*E=0$ by the dimension
axiom.  Transposition therefore yields
\begin{equation}\label{eq:murre-endpoint-orthogonality}
  E\pi_j=\pi_jE=Q\pi_j=\pi_jQ=0
  \qquad(j=0,4).
\end{equation}

Since $X$ is a surface, put
\[
  \gamma={}^t\!D_\beta\circ D_\beta
  =(p_{13})_*\bigl(p_{12}^*D_\beta\cdot
                    p_{23}^*{}^t\!D_\beta\bigr)
  \in\CH^0(X\times X)=\mathbb Z[X\times X].
\]
Let $\iota=e\times\id_X$ and let $q_1,q_2:X\times X\to X$ be the
projections.  Base change along $\iota$ identifies the fiber product
of $p_{13}$ with $X\times X$ via $(y,z)\mapsto(e,y,z)$, and gives
\[
  \iota^*\gamma
  =(q_2)_*\left(
      q_1^*\bigl(\iota^*D_\beta\bigr)
      \cdot{}^t\!D_\beta
    \right)=0.
\]
Indeed, $\iota^*D_\beta=0$ by the normalization of the Poincar\'e
bundle.  The map $\iota^*:\CH^0(X\times X)=\mathbb Z\to
\CH^0(X)=\mathbb Z$ sends generator to generator, so $\gamma=0$.
Consequently,
\[
  Q\circ E
  =\Gamma_i\circ{}^t\!\Gamma_i\circ\gamma
       \circ\Gamma_i\circ{}^t\!\Gamma_i
  =0.
\]
This is the integral form of~\cite[Lemmas~6.2--6.3]{Murre1990}.

Set
\[
  \pi_1=E,\qquad \pi_3=Q-EQ,\qquad
  \pi_2=[\Delta_X]-\pi_0-\pi_1-\pi_3-\pi_4.
\]
Since $QE=0$, one has
$(Q-EQ)^2=Q-EQ$ and $E(Q-EQ)=(Q-EQ)E=0$.  Together with
\eqref{eq:murre-endpoint-orthogonality}, this proves that the five
$\pi_j$ are orthogonal idempotents summing to $\Delta_X$.
Finally, $h_\beta\Pic^0(i)=\beta\Psi_i=\id$ and
\eqref{eq:murre-F-normalization} makes $F_\beta$ factor through
$\pi_1^C=\Delta_C-[c\times C]-[C\times c]$.  Hence $E$ acts only as
the identity on $H^1(X)$; by transposition, $Q$ acts only as the
identity on $H^3(X)$, so $EQ$ is homologically trivial.  The remaining
projectors act in degrees $0,2,4$, as in
\cite[Sections~4--6]{Murre1990}.
\end{proof}

\begin{theorem}\label{PPAV}
Let $(A,\Theta)$ be a principally polarized abelian surface over an
algebraically closed field of characteristic zero.  There are mutually
orthogonal idempotents
\[
  \pi_i\in\CH^2(A\times A),\qquad 0\le i\le4,
\]
whose sum is $[\Delta_A]$, with
$\pi_0=[e\times A]$ and $\pi_4=[A\times e]$.  For every Weil
cohomology theory, $\pi_i$ acts as the identity on $H^i(A)$ and as zero
on $H^j(A)$ for $j\ne i$.  Consequently,
\[
  h(A)\simeq\bigoplus_{i=0}^4(A,\pi_i)
\]
in the category of integral Chow motives.
\end{theorem}

\begin{proof}
The theta cycle generates $A$ and satisfies $\Theta^2=2=\dim A$.
The Matsusaka--Ran criterion
\cite[Theorem, pp.~329--331]{d00284d3-087e-3a08-8909-ade837d4b99e}
therefore identifies $(A,\Theta)$ either with the product of two
canonically polarized elliptic curves or with $(J(C),\Theta_C)$ for a
smooth genus-two curve $C$.  In the first case, tensor the integral
three-projector decompositions of the two curves and group by degree.

In the second case, translate $C$ to obtain a pointed embedding
$j:(C,c)\hookrightarrow(A,e)$, and let
$a:J(C)\xrightarrow{\simeq}A$ be the induced isomorphism.  If
$\lambda_C$ and $\lambda_\Theta$ denote the principal polarizations,
then
\[
  a^\vee\lambda_\Theta a=\lambda_C.
\]
Under $\Pic^0(A)=A^\vee$ and $\operatorname{Alb}(A)=A$, the homomorphism
in Lemma~\ref{lem:integral-murre-curve} is therefore
\[
  \Psi_j=a\lambda_C^{-1}a^\vee
  =\lambda_\Theta^{-1}:A^\vee\xrightarrow{\simeq}A.
\]
It is an isomorphism, so the lemma supplies the required integral
projectors.
\end{proof}

\begin{remark}
The three middle projectors need not lie in $\im(u)$; a complete
Milnor--Witt lift therefore still requires the image hypothesis of
Corollary~\ref{simultaneousprojectors}.
\end{remark}

\bibliographystyle{alphaurl}
\bibliography{Reference}

\end{document}